\documentclass[11pt]{article}

\usepackage[letterpaper,margin=1in]{geometry}
\usepackage[T1]{fontenc}
\usepackage{lmodern}
\usepackage{microtype}
\usepackage{amsmath,amssymb,amsthm,mathtools,bm}
\usepackage{enumitem}
\usepackage[round,authoryear]{natbib}
\usepackage{mathrsfs}
\usepackage{booktabs}
\usepackage{xcolor}
\usepackage[
  colorlinks=true,
  citecolor=red,
  linkcolor=red,
  urlcolor=red
]{hyperref}
\hypersetup{
  pdftitle={An Unrestricted Cubic-Root Gaussian Approximation Bound for Hyperrectangles: Disproving the Polynomial-Dimensional n-1/4 Near-Optimality Conjecture},
  pdfauthor={ChatGPT 5.6 Pro (OpenAI)}
}

\allowdisplaybreaks
\setlist[enumerate]{leftmargin=2.2em}

\newtheorem{theorem}{Theorem}
\newtheorem{assumption}{Assumption}
\newtheorem{lemma}{Lemma}
\newtheorem{proposition}{Proposition}

\newtheorem{definition}{Definition}

\newtheorem{example}{Example}

\title{Cubic-Root Gaussian Approximation under Unrestricted Covariance}
\author{
Zijun Gao\thanks{Department of Data Sciences and Operations, Marshall School of Business, University of Southern California, Los Angeles, CA, USA, \href{zijungao@marshall.usc.edu}{zijungao@marshall.usc.edu}}
\and
Weihan Zhang\thanks{Division of Biostatistics, School of Public Health, University of California, Berkeley, Berkeley, CA, USA, \href{weihan_zhang2001@berkeley.edu}{weihan\_zhang2001@berkeley.edu}}
}

\begin{document}
\maketitle

\begin{abstract}
For Gaussian approximation over high-dimensional rectangles under unrestricted covariance, \citet{CCK2023} conjectured that the \(n^{-1/4}\) rate, up to logarithmic factors, is near-optimal. We show that, under the coordinatewise subexponential condition with scale \(B_n\) and the marginal variance lower bound condition with constant \(b\) in \citet{CCK2023}, 
the approximation error in dimension \(d\) is bounded by
\begin{align*}
C_b\min\left\{
1,\,
\left(\frac{B_n^2}{n}\right)^{1/3}\{\log(2dn)\}^{7/3}
+
\frac{B_n}{\sqrt n}\{\log(2dn)\}^{5/2}
\right\}.
\end{align*} 
In particular, for bounded \(B_n\) and polynomial dimension, the new bound is \(n^{-1/3}\) and therefore falsifies the polynomial-dimensional \(n^{-1/4}\) near-optimality conjecture. The proof uses a two-stage interpolation and a rank-free matrix-weighted Gaussian surface bound, which may be of independent interest.

The initial proof attempt was generated by ChatGPT 5.6 Pro (OpenAI) and subsequently corrected and rewritten by the authors. The machine-checked Lean formalization of the proof can be found at the \href{https://github.com/WeihanZhang2001/cubic-root-gaussian-approximation-under-unrestricted-covariance}
{GitHub repository}.
\end{abstract}

\section{Introduction and Main Results}\label{sec:intro}
Let \(X_1,\ldots,X_n\in\mathbb R^d\) be independent centered random vectors.
For \(a=(a_1,\ldots,a_d)^\top\in\mathbb R^d\), define
\begin{align*}
T_n(a)
&=
\max_{1\leq j\leq d}
\left\{
\frac{1}{\sqrt n}\sum_{i=1}^nX_{ij}+a_j
\right\}.
\end{align*}
We state two assumptions on \(X_i\).

\begin{assumption}[Subexponential tail]
\label{assu:sub.exp}
There exists \(B_n>0\), possibly depending on \(n\), such that
\begin{align*}
\mathbb E\left[\exp\left\{\frac{|X_{ij}|}{B_n}\right\}\right]
&\leq
2, \quad \forall~i \in [n], ~~\forall~j \in [d].
\end{align*}
\end{assumption}

\begin{assumption}[Non-vanishing marginal variance]
\label{assu:marginal-variance-lower-bound}
There exists a universal constant \(b>0\) such that
\begin{align*}
\frac{1}{n}\sum_{i=1}^n\mathbb E\left[X_{ij}^2\right]
&\geq
b^2, \quad \forall~j \in [d].
\end{align*}
\end{assumption}

Let \(\Sigma_n=n^{-1}\sum_{i=1}^n\mathbb E[X_iX_i^\top]\), which is finite under Assumption~\ref{assu:sub.exp}. Consider the Gaussian approximation to \(T_n(a)\),
\begin{align*}
T_n^G(a)
&=
\max_{1\leq j\leq d}(Z_j+a_j),
\quad
Z\sim N(0,\Sigma_n).
\end{align*}
For \(a\in\mathbb R^d\), define $F_{n,a}^G(t)
:=
\mathbb P\{T_n^G(a)\le t\}$, $t\in\mathbb R$. For \(\alpha\in(0,1)\), define the lower
\((1-\alpha)\)-quantile by
\begin{align}
c_{n,1-\alpha}^G(a)
:=
\inf\left\{
t\in\mathbb R:
F_{n,a}^G(t)\ge1-\alpha
\right\}.
\label{eq:lower-gaussian-quantile}
\end{align} The following theorem characterizes the accuracy of the Gaussian critical value.

\begin{theorem}
\label{thm:subexponential-gaussian-approximation}
Under Assumption~\ref{assu:sub.exp} and Assumption~\ref{assu:marginal-variance-lower-bound}, there exists \(C_b\in(0,\infty)\) depending only on \(b\), such that, uniformly over all \(n,d\), all distributions satisfying the assumptions, every \(a\in\mathbb R^d\) and \(\alpha\in(0,1)\),
\begin{align}
\left|
\mathbb P\left\{T_n(a)>c_{n,1-\alpha}^G(a)\right\}-\alpha
\right|
&\leq
C_b
\min\left\{
1,
\left(\frac{B_n^2}{n}\right)^{1/3}\{\log(2dn)\}^{7/3}
+
\frac{B_n}{\sqrt n}\{\log(2dn)\}^{5/2}
\right\}.
\label{eq:subexponential-gaussian-critical-value-bound}
\end{align}
\end{theorem}

Assumption~\ref{assu:sub.exp} rules out heavy-tailed coordinates, which is the  Condition E in \citet{CCKK2022}.
Assumption~\ref{assu:marginal-variance-lower-bound} prevents the marginal variances from vanishing, as shrinking a non-Gaussian coordinate shrinks \(B_n\) and the upper bound but not necessarily the approximation error (see Example~\ref{exam:counter.variance.lower.bound} for an illustration).
Assumption~\ref{assu:marginal-variance-lower-bound} is related to the marginal-variance component of Condition M in \citet{CCKK2022}, but {\color{black}their} Condition M additionally assumes
\(n^{-1}\sum_{i=1}^n\mathbb E[X_{ij}^4]\leq B_n^2b_2^2\).
We highlight that no further structural condition is imposed on \(\Sigma_n\). In particular, the covariance matrix may be degenerate.

We compare Theorem~\ref{thm:subexponential-gaussian-approximation} with the open conjecture result in \citet{CCKK2022}. 
Under the additional fourth-moment condition imposed by \citet{CCKK2022}, their unrestricted-covariance bound is of order \((B_n^2/n)^{1/4}\{\log(2dn)\}^{5/4}\). The ratios of the two terms in Theorem~\ref{thm:subexponential-gaussian-approximation} to this quarter-power bound are
\begin{align*}
\frac{(B_n^2/n)^{1/3}\{\log(2dn)\}^{7/3}}
{(B_n^2/n)^{1/4}\{\log(2dn)\}^{5/4}}
&=
\left\{
\frac{B_n^2\{\log(2dn)\}^{13}}{n}
\right\}^{1/12},
\\
\frac{(B_n/\sqrt n)\{\log(2dn)\}^{5/2}}
{(B_n^2/n)^{1/4}\{\log(2dn)\}^{5/4}}
&=
\left(\frac{B_n^2}{n}\right)^{1/4}
\{\log(2dn)\}^{5/4}.
\end{align*}
Thus, \(B_n^2\{\log(2dn)\}^{13}/n\to0\) implies that both ratios converge to zero, and hence that the present bound is of smaller order. In particular, for a fixed $B_n$ and polynomial dimension, the conjecture requires a worst-case lower bound of order \(n^{-1/4}\), whereas Theorem~\ref{thm:subexponential-gaussian-approximation} gives a uniform upper bound of order \(n^{-1/3}\), thereby disproving the polynomial-dimensional \(n^{-1/4}\) near-optimality conjecture.
Table~\ref{tab:literature} provides a more comprehensive comparison {\color{black}of} Theorem~\ref{thm:subexponential-gaussian-approximation} with existing results, and further literature review can be found in Section~\ref{appe:sec:literature} of the appendix.

\begin{table}[tbp]
\centering
\caption{Summary of selected Gaussian approximation results related to
Theorem~\ref{thm:subexponential-gaussian-approximation}. ``Approximation rate'' is stated for bounded \(B_n\) and polynomial dimension up to
logarithmic factors.
``Unrestricted covariance'' indicates whether singular covariance is allowed. 
For {\color{black}``Distribution-free''}, moment and tail conditions are not regarded as additional distributional structure.}
\label{tab:literature}
\small
\begin{tabular}{p{8cm}ccc}
\toprule
Results
&
\shortstack{Unrestricted\\covariance}
&
\shortstack{Distribution\\-free}
&
\shortstack{Approximation\\rate}
\\
\midrule

\citet{FangKoike2021}; \citet{KuchibhotlaRinaldo2020};
\citet{Lopes2022}; \citet{CCK2023}
&
\(\times\)
&
\(\checkmark\)
&
\(n^{-1/3}\) to \(n^{-1/2}\)
\\
\citet{FangKoikeLiuZhao2023}
&
\(\times\)
&
\(\checkmark\)
&
\(n^{-1/2}\)
\\
\citet{FangKoike2024}
&
\(\checkmark\)
&
\(\times\) (iid log-concave)
&
\(n^{-1/2}\)
\\
\citet{CCK2017}; \citet{Koike2021}
&
\(\checkmark\)
&
\(\checkmark\)
&
\(n^{-1/6}\)
\\
\citet{CCKK2022}
&
\(\checkmark\)
&
\(\checkmark\)
&
\(n^{-1/4}\)
\\
Theorem~\ref{thm:subexponential-gaussian-approximation}
&
\(\boldsymbol{\checkmark}\)
&
\(\boldsymbol{\checkmark}\)
&
\(\boldsymbol{n^{-1/3}}\)
\\
\bottomrule
\end{tabular}
\end{table}

\paragraph{Contributions.}
We give an unrestricted-covariance upper bound with polynomial order \(n^{-1/3}\) in the fixed or moderately growing-envelope and polynomial-dimensional regime, providing the best known rate in this regime and disproving the polynomial-dimensional \(n^{-1/4}\) near-optimality conjecture in \citet{CCK2023}.
For the proof, we highlight three ingredients.
\begin{itemize}
    \item To bound the smoothed Gaussian approximation error, we consider a two-stage interpolation and leverage different properties along the two paths.
    
    \item In the first interpolation, moment matching through order three eliminates all terms through order three.
    We use leave-one-out ghost completion and the convex-partition bound to control the resulting fourth-order and higher-order terms.
    
    \item In the second interpolation, a third-order term remains, and we leverage the uniformly nondegenerate Gaussian component {\color{black}to} control it using a rank-free matrix-weighted Gaussian surface bound, which {\color{black}may also be} of independent interest.
\end{itemize}

\paragraph{Organization.}
Section~\ref{sec:preparation} provides details of the preprocessing steps, including truncation, marginal standardization, and support restriction. 
Section~\ref{sec:smooth-selector} introduces a smooth approximation to the maximum event probability. Section~\ref{sec:interpolation} analyzes the two interpolation paths to bound the smooth probability difference. 
Section~\ref{sec:summary} combines the bounds and chooses the hyperparameters to complete the proof. Supporting results, counterexamples, and the literature review are deferred to the appendix.

\paragraph{Notations.}
We write $\mathcal R^d$ for the class of axis-aligned rectangles $\prod_{j=1}^d I_j$, where each $I_j\subseteq\mathbb R$ is an interval.
For random vectors $V,W\in\mathbb R^d$, define their rectangle distance by
$
\rho_{\mathcal R}(V,W)
=
\sup_{A\in\mathcal R^d}
\left|\mathbb P(V\in A)-\mathbb P(W\in A)\right|$.
For a covariance matrix $\Sigma \in \mathbb R^{d\times d}$,
let
\(
\ker(\Sigma)
=
\{v\in\mathbb R^d: \Sigma v=0\}
\)
denote the null space of $\Sigma$, and let
\(
\operatorname{range}(\Sigma)
=
\{\Sigma v:v\in\mathbb R^d\}
=
\ker(\Sigma)^\perp
\)
denote its column space.
Let $\otimes$ denote the tensor product (for
${\color{black}z=(z_1,\ldots,z_r)^\top}\in\mathbb R^r$ and an integer $k\geq1$,
$z^{\otimes k}$ denotes the $k$-fold tensor product of $z$, with entries
$
(z^{\otimes k})_{a_1,\ldots,a_k}
=
\prod_{\ell=1}^k z_{a_\ell}$, $
a_1,\ldots,a_k\in\{1,\ldots,{\color{black}r}\}$).
We write $\varphi_E$ and $\gamma_E$ for the standard Gaussian density and measure on a Euclidean subspace $E$, and use $\varphi_r$ and $\gamma_r$ when $E=\mathbb R^r$.

\section{Preprocessing}
\label{sec:preparation}

\subsection{Truncation and recentering}

Let \(A_i=\{\|X_i\|_\infty\leq32B_n\log(2dn)\}\). Let \(X_{1,A_1},\ldots,X_{n,A_n}\) be independent with \(X_{i,A_i}\sim\mathbb P(X_i\mid A_i)\), and set \(\widetilde X_i=X_{i,A_i}-\mathbb E[X_{i,A_i}]\) and \(\widetilde\Sigma_n=n^{-1}\sum_{i=1}^n\mathbb E[\widetilde X_i\widetilde X_i^\top]\). The preprocessing lemma in
Appendix~\ref{sec:preparation-appendix} gives the following bounds,
without restricting the parameter regime, for a universal constant
\(C>0\):
\begin{align}\label{eq:preprocessing-errors-main}
\begin{split}
\sum_{i=1}^n\mathbb P(A_i^c)
&\leq
(2dn)^{-31}
\leq
n^{-20},\\
\left\|
\frac{1}{\sqrt n}
\sum_{i=1}^n\mathbb E[X_{i,A_i}]
\right\|_\infty
&\leq
C\sqrt n\,dB_n\{1+\log(2dn)\}(2dn)^{-32},\\
\|\widetilde\Sigma_n-\Sigma_n\|_\infty
&\leq
CdB_n^2\{1+\{\log(2dn)\}^2\}(2dn)^{-32}.
\end{split}
\end{align}
Moreover, \(\|\widetilde X_i\|_\infty\leq64B_n\log(2dn)\) a.s.

Write
\[
\mathfrak r_{n,d}
:=
\left(\frac{B_n^2}{n}\right)^{1/3}
\{\log(2dn)\}^{7/3}
+
\frac{B_n}{\sqrt n}
\{\log(2dn)\}^{5/2}.
\]
Choose \(c_b>0\) sufficiently small and \(N_b\in\mathbb N\)
sufficiently large, both depending only on \(b\). We refer to
\begin{align}
\mathfrak r_{n,d}\leq c_b,
\qquad
n\geq N_b
\label{eq:preprocessing-regime}
\end{align}
as the nontrivial preprocessing regime.

If \(\mathfrak r_{n,d}>c_b\), the theorem follows from the trivial
probability bound after increasing \(C_b\). If \(n<N_b\), the finitely
many remaining values of \(n\) are likewise absorbed into \(C_b\);
here \(B_n\geq b/2\) makes the target rate uniformly bounded below
over \(d\). Thus it remains to work under
Eq.~\eqref{eq:preprocessing-regime}.

In this regime, Eq.~\eqref{eq:preprocessing-errors-main} gives
\[
\left\|
\frac{1}{\sqrt n}
\sum_{i=1}^n\mathbb E[X_{i,A_i}]
\right\|_\infty
+
\|\widetilde\Sigma_n-\Sigma_n\|_\infty
\leq
n^{-10}.
\]

\subsection{Marginal variance standardization}

Under Eq.~\eqref{eq:preprocessing-regime}, the estimates proved in
Appendix~\ref{sec:preparation-appendix} give
\[
(\widetilde\Sigma_n)_{jj}\geq b^2/2,
\qquad
1\leq j\leq d.
\]
Consequently, the marginal standardization is well-defined. Let
\[
D_n
:=
\operatorname{diag}
\left\{
(\widetilde\Sigma_n)_{jj}^{-1/2}
:
1\leq j\leq d
\right\},
\]
and replace \(\widetilde X_i\) and \(\widetilde\Sigma_n\) by
\(D_n\widetilde X_i\) and
\(D_n\widetilde\Sigma_nD_n\), respectively. With a slight abuse of
notation, we continue to use the same symbols for the standardized
objects. By Lemma~\ref{lem:positive-diagonal-invariance}, this
standardization leaves the rectangle distance unchanged. We then have
\[
(\widetilde\Sigma_n)_{jj}=1,
\qquad
\max_{1\leq i\leq n}\|\widetilde X_i\|_\infty
\leq
\sqrt n\,b_{d,n},
\qquad
b_{d,n}
:=
\frac{64\sqrt2B_n\log(2dn)}{b\sqrt n}.
\]

\begin{lemma}
\label{lem:preprocessing-transfer-main}
Under the preprocessing regime
\eqref{eq:preprocessing-regime}, there exists a constant \(C_b>0\),
depending only on \(b\), such that, with $W=n^{-1/2}\sum_{i=1}^nX_i$, $Z\sim N(0,\Sigma_n)$, and with \(\widetilde W\) and
\(\widetilde Z\sim N(0,\widetilde\Sigma_n)\) denoting the corresponding
objects after the final marginal standardization,
\[
\sup_{A\in\mathcal R^d}
\left|\mathbb P(W\in A)-\mathbb P(Z\in A)\right|
\leq
\sup_{A\in\mathcal R^d}
\left|
\mathbb P(\widetilde W\in A)
-
\mathbb P(\widetilde Z\in A)
\right|
+
n^{-20}
+
C_b n^{-4}.
\]
\end{lemma}

The proof of Lemma~\ref{lem:preprocessing-transfer-main} is given in Appendix~\ref{sec:preparation-appendix}. 

\subsection{Support restriction}
\label{subsec:support-restriction}

Since $\widetilde\Sigma_n$ may be singular, we restrict to  $\operatorname{range}(\widetilde\Sigma_n)$ and whiten only on this support. Let $r=\operatorname{rank}(\widetilde\Sigma_n)$ and factor $\widetilde\Sigma_n=UU^\top$, where $U\in\mathbb R^{d\times r}$ has full column rank. Write $u_j^\top$ for the $j$th row of $U$, so $\|u_j\|_2^2=(\widetilde\Sigma_n)_{jj}=1$. If $v\in\ker(\widetilde\Sigma_n)$, then $\sum_i\mathbb E\{v^\top\widetilde X_i/\sqrt n\}^2=0$, and hence $v^\top\widetilde X_i=0$ a.s. for every fixed $i$ and thus $\widetilde X_i/\sqrt n\in\ker(\widetilde\Sigma_n)^\perp=\operatorname{range}(U)$ a.s. With $U^+=(U^\top U)^{-1}U^\top$, define $Y_i=U^+(\widetilde X_i/\sqrt n)$. Then
\begin{align*}
UY_i
=
\frac{\widetilde X_i}{\sqrt n},
\quad
\sum_{i=1}^n\mathbb E[Y_iY_i^\top]
=
I_r,
\quad
|\langle u_j,Y_i\rangle|
=
\frac{|\widetilde X_{ij}|}{\sqrt n}
\leq {\color{black}b_{d,n}}.
\end{align*}
Although $Y_i$ may have a large Euclidean norm when $\widetilde\Sigma_n$ has small positive eigenvalues, the proof uses only the bound on $|\langle u_j,Y_i\rangle|$ and therefore imposes no lower bound on the smallest positive eigenvalue.

Theorem~\ref{thm:subexponential-gaussian-approximation} investigates maximum statistics, whose distributional events correspond to axis-aligned rectangles in $\mathbb R^d$. An empty rectangle has probability zero under both laws and hence
contributes zero to the rectangle distance. By
Lemma~\ref{lem:closed-rectangles}, it is therefore sufficient to
consider nonempty closed rectangles whose endpoints are all finite. Complete the normal list by setting $u_{d+j}=-u_j$ for $j=1,\ldots,d$, which spans $\mathbb R^r$. Under $x=Uy$, the rectangle $\prod_{j=1}^d[a_j^-,a_j^+]$ is mapped to
\begin{align*}
\bigcap_{j=1}^d\{y\in\mathbb R^r:\langle u_j,y\rangle\leq a_j^+\}
\cap
\bigcap_{j=1}^d\{y\in\mathbb R^r:\langle u_{d+j},y\rangle\leq-a_j^-\},
\end{align*}
a polyhedron described by at most $2d$ halfspaces.

\section{Smooth approximation of maximum event probability}
\label{sec:smooth-selector}

To prove Theorem~\ref{thm:subexponential-gaussian-approximation}, it is sufficient to establish the same upper bound for the Kolmogorov distance between the maximum statistics,
\begin{align*}
\sup_{a\in\mathbb R^d,\;t\in\mathbb R}
\left|
\mathbb P\{T_n(a)\leq t\}
-
\mathbb P\{T_n^G(a)\leq t\}
\right|.
\end{align*}
The maximum statistic and the indicator of its corresponding polyhedral event are not smooth, which prevents direct Taylor expansion along an interpolation path. We therefore use the softmax selector, a standard smooth approximation to the maximum used in high-dimensional Gaussian approximation (see, for example, \citet{CCK2017}). 

For \(a_0\in\mathbb R\),
\(a=(a_1,\ldots,a_{2d})\in\mathbb R^{2d}\), and \(\beta>0\), let $P_a
=
\bigcap_{j=1}^{2d}
\left\{
y\in\mathbb R^r:
\langle u_j,y\rangle\leq a_j
\right\}$. Set $s_0(y)=a_0$ and $s_j(y)=\langle u_j,y\rangle-a_j$, for $j=1,\ldots,2d$, and define
\begin{align}
\pi_j(y)
=
\frac{\exp\{\beta s_j(y)\}}{\sum_{\ell=0}^{2d}\exp\{\beta s_\ell(y)\}},
\quad j=0,\ldots,2d.
\label{eq:selector-definition}
\end{align}
We denote $\pi_0(y)$ by $\pi_{\beta,a_0,a}(y)$.
Here $\beta$ is the inverse smoothing bandwidth: a larger $\beta$ gives a sharper approximation to the indicator function $\mathbf 1_{P_a}$, whereas a smaller $\beta$ makes the selector smoother.

\begin{proposition}
\label{prop:probability-comparison-through-selectors}
Fix $0<\delta\leq1/4$, let $h=\log(8d/\delta)$, and $a_0^+=h/\beta$, $a_0^-=-h/\beta$. Let $S,G\in\mathbb R^r$ be random vectors and suppose that
\begin{align}
\sup_{a\in\mathbb R^{2d}, a_0 \in \{a_0^-,a_0^+\}}
\left|
\mathbb E\pi_{\beta,a_0,a}(S)
-
\mathbb E\pi_{\beta,a_0,a}(G)
\right|
&\leq\Delta.
\label{eq:selector-comparison-hypothesis}
\end{align}
Then there exists a universal constant $C>0$ such that 
\begin{align*}
\sup_{a\in\mathbb R^{2d}}
|\mathbb P(S\in P_a)-\mathbb P(G\in P_a)|
&\leq
C\left[
\Delta+\delta+
\sup_{\substack{a\in\mathbb R^{2d}\\t\in\mathbb R}}
\mathbb P\left
\{
t<\max_{1\leq j\leq2d}\{\langle u_j,G\rangle-a_j\}
\leq t+\frac{2h}{\beta}
\right\}
\right].
\end{align*}
\end{proposition}

Proposition~\ref{prop:probability-comparison-through-selectors} separates the polyhedral probability error into the smooth expectation difference $\Delta$, the selector approximation error $\delta$, and a Gaussian boundary layer of width $2h/\beta$. The first term is bounded via interpolation in Section~\ref{sec:interpolation}. The boundary layer is controlled by the conditional form of Nazarov's inequality in Lemma~\ref{lem:conditional-nazarov}, based on \citet{CCK2015}.

\section{Smooth probability comparison via path interpolation}
\label{sec:interpolation}

Let $S=\sum_{i=1}^nY_i$ and let $G\sim N(0,I_r)$. A direct interpolation from $S$ to $G$ leaves a third-order term, while the independent Gaussian component at the initial endpoint $S$ vanishes. Therefore, we introduce an intermediate random vector
\begin{align*}
T
=
\frac{1}{\sqrt2}G+2\sum_{i=1}^nD_iY_i',
\quad
D_i
\sim
\operatorname{Bernoulli}(1/8),
\end{align*}
where \((Y_i')_{i=1}^n\) is an independent copy of the entire family
\((Y_i)_{i=1}^n\), the variables \(D_1,\ldots,D_n\) are mutually
independent, and the four random elements
\[
G,\qquad (Y_i)_{i=1}^n,\qquad (Y_i')_{i=1}^n,\qquad
(D_i)_{i=1}^n
\]
are mutually independent. Direct calculation in Lemma~\ref{lem:path-moments} gives
\begin{align*}
\mathbb E[T]
=\mathbb E[S]=0,
\quad
\operatorname{Cov}(T)
=\operatorname{Cov}(S)=I_r,
\quad
\mathbb E[T^{\otimes3}]
=\mathbb E[S^{\otimes3}].
\end{align*}
Thus $T$ matches the first three moments of $S$ and contains the fixed Gaussian component $G/\sqrt2$. 
% Because the post-truncation envelope ${\color{black}b_{d,n}}$ may depend on $n$ and $d$, the Taylor expansion is carried to an even order $Q$ that is allowed to grow with $n$, which makes the terminal term negligible. 
We decompose the target difference as
\begin{align*}
\mathbb E\pi_{\beta,a_0,a}(S)-\mathbb E\pi_{\beta,a_0,a}(G)
&=
\underbrace{\{\mathbb E\pi_{\beta,a_0,a}(S)-\mathbb E\pi_{\beta,a_0,a}(T)\}}_{S \to T}
+
\underbrace{\{\mathbb E\pi_{\beta,a_0,a}(T)-\mathbb E\pi_{\beta,a_0,a}(G)\}}_{T \to G}.
\end{align*}
The first difference regarding $S\to T$ is bounded in Section~{\color{black}\ref{sec:S-to-T}}, and the second difference regarding $T\to G$ is bounded in Section~\ref{subsec:T-to-G}.

\subsection{Third moment preserving interpolation ($S\to T$)}
\label{sec:S-to-T}

\begin{proposition}
\label{prop:S-to-T}
There exist universal constants $C,C^*>0$ such that, for every even integer $Q\geq8$, every $\beta>0$ satisfying $C^*Q\beta {\color{black}b_{d,n}}\leq1/4$, and every $0<\tau<1$,
\begin{align*}
|\mathbb E\pi_{\beta,a_0,a}(S)-\mathbb E\pi_{\beta,a_0,a}(T)|
&\leq
C\left\{
n{\color{black}b_{d,n}}^4\beta^4\tau
+
{\color{black}b_{d,n}}^2\beta^2\log(4d)\log(1/\tau)
+
nQ^2(C^*Q\beta {\color{black}b_{d,n}})^Q
\right\}.
\end{align*}
\end{proposition}

\begin{proof}[Proof of Proposition~\ref{prop:S-to-T}]
Let \(U_1,\ldots,U_n\) be independent
\(\operatorname{Uniform}(0,1)\) random variables, jointly independent
of \(G\), \((Y_i)_{i=1}^n\), and \((Y_i')_{i=1}^n\). For
\(0\leq s\leq1\), define $B_i(s)
=
\mathbf 1\!\left\{
U_i\leq(1-s/2)^3
\right\}$, for $i=1,\ldots,n$, and define
\begin{align*}
S_s
&=
\sqrt{\frac{s}{2}}G
+
\frac{1}{1-s/2}\sum_{i=1}^nB_i(s)Y_i.
\end{align*}
Then $S_0=S$ and $S_1\stackrel d=T$.

For every $s\in[0,1]$, Lemma~\ref{lem:path-moments} gives
\begin{align*}
\mathbb{E}[S_s]
=0,
\quad
\operatorname{Cov}(S_s)
=I_r,
\quad
\mathbb E[S_s^{\otimes3}]
=\mathbb E[S^{\otimes3}].
\end{align*}
Because the moments through order three are constant along the path, Lemma~\ref{lem:exact-Bernoulli-generator} gives
\begin{align}
&\frac{d}{ds}\mathbb E\pi_{\beta,a_0,a}(S_s)
={}
\frac{(1-s/2)^2}{4}\sum_{i=1}^n\mathbb E\int_0^1t(1-t)^2 \underbrace{D^4\pi_{\beta,a_0,a}\left(R_{i,s}+t\frac{Y_i}{1-s/2}\right)\left[\left(\frac{Y_i}{1-s/2}\right)^{\otimes4}\right]}_{:=f_{i,t}^{(1)}(R_{i,s})}dt
\nonumber\\
&+
\frac{(1-s/2)^5}{4}\sum_{i=1}^n\mathbb E\int_0^1(1-t)\underbrace{D^4\pi_{\beta,a_0,a}\left(R_{i,s}+t\frac{Y_i}{1-s/2}\right)
\left[\frac{Y_i}{1-s/2},\frac{Y_i}{1-s/2},\frac{Y_i'}{1-s/2},\frac{Y_i'}{1-s/2}\right]}_{:=f_{i,t}^{(2)}(R_{i,s})}dt,
\label{eq:main-Bernoulli-generator}
\end{align}
where $
R_{i,s}
:=
\sqrt{\frac{s}{2}}G
+
\frac{1}{1-s/2}\sum_{\ell\neq i}B_\ell(s)Y_\ell$.

For $s\in[0,\tau]$, the exact selector derivative bound in Lemma~\ref{lem:selector-derivatives} gives
\begin{align*}
|D^4\pi_{\beta,a_0,a}(y)[z_1,z_2,z_3,z_4]|
&\leq
C\beta^4\prod_{\ell=1}^4\max_{j\leq2d}|\langle u_j,z_\ell\rangle|.
\end{align*}
Since $1-s/2\geq1/2$ and $\max_j|\langle u_j,Y_i\rangle|\leq {\color{black}b_{d,n}}$, integrating Eq.~\eqref{eq:main-Bernoulli-generator} over $[0,\tau]$ yields
\begin{align}
\left|\int_0^\tau\frac{d}{ds}\mathbb E\pi_{\beta,a_0,a}(S_s)ds\right|
&\leq
Cn{\color{black}b_{d,n}}^4\beta^4\tau.
\label{eq:small-s-main-bound}
\end{align}

We next deal with $s\in(\tau,1]$. 
Note that $R_{i,s}$ in Eq.~\eqref{eq:main-Bernoulli-generator} is observation-specific. We add an independent ghost copy of the omitted increment to $R_{i,s}$ to recover a common full sum law. Explicitly, let $\eta_i=B_i(s)Y_i/(1-s/2)$, $\eta_i^{(0)}$ be an independent copy of $\eta_i$, and define $H_i:=R_{i,s}+\eta_i^{(0)}$.
Conditional on the active directions \(Y_i,Y_i'\), let \(P_i\)
denote convolution in the spatial variable by the law of the
independent ghost increment \(\eta_i^{(0)}\). Thus \(P_i\) leaves
\(Y_i,Y_i'\) fixed. Lemma~\ref{lem:finite-ghost-completion} gives 
\begin{align}
\mathbb E\left[
f_{i,t}^{(\ell)}(R_{i,s};Y_i,Y_i')
\right]
&=
\sum_{k=0}^{(Q-4)/2-1}
\mathbb E\left[
\bigl[(I-P_i)^k f_{i,t}^{(\ell)}\bigr]
(H_i;Y_i,Y_i')
\right]
\\
&\quad+
\mathbb E\left[
\bigl[(I-P_i)^{(Q-4)/2} f_{i,t}^{(\ell)}\bigr]
(R_{i,s};Y_i,Y_i')
\right].
\end{align}
Using Definition~\ref{def:common-terminal-contributions}, the exact
ghost and Taylor decomposition of the generator is
\begin{align}
\frac{d}{ds}
\mathbb E\pi_{\beta,a_0,a}(S_s)
&=
\mathcal N_Q(s)+\mathcal T_Q(s),
\qquad s\in[\tau,1].
\label{eq:NQ-TQ-decomposition}
\end{align}
The quantity \(\mathcal N_Q(s)\) contains all explicit Taylor terms
of derivative order below \(Q\), while \(\mathcal T_Q(s)\) contains
the exact order-\(Q\) Taylor remainders and the terminal ghost
remainders. Here $H_i$ has the common full-sum law
% \begin{align*}
% \sqrt{\frac{s}{2}}G+\sum_{j=1}^n\frac{B_j(s)Y_j}{a_s}
% \end{align*}
and is independent of $Y_i/(1-s/2)$, $Y'_i/(1-s/2)$ and all further ghost copies used by $(I-P_i)^k$. Therefore, $H_i$ may be recoupled as one independent common state before summing over $i$.

Every common-base term retains a repeated direction pair. Conditional on the non-Gaussian part of the common state, two Gaussian integrations by parts in that pair give, for $q<Q$,
\begin{align}
\mathbb E_G\left[D^q\pi_{\beta,a_0,a}\left(H+\sqrt{\frac{s}{2}}G\right)[z,z,z_3,\ldots,z_q]\right]
=&
\frac{2}{s} \mathbb E_G\left[\{(G^\top z)^2-\|z\|_2^2\}
\nonumber\right.\\
&\times\left.
D^{q-2}\pi_{\beta,a_0,a}\left(H+\sqrt{\frac{s}{2}}G\right)[z_3,\ldots,z_q]\right].
\label{eq:two-IBP-main}
\end{align}
The selector derivative formula in
Lemma~\ref{lem:selector-derivatives} and the Gumbel representation in
Lemma~\ref{lem:Gumbel-cells} then convert the remaining softmax
products into probabilities of convex winner cells. Applying the convex cell bound in
Lemma~\ref{lem:convex-cell-contraction} {\color{black}gives} the upper bound of the total contribution of the derivative of order $q < Q$ 
\begin{align}
\frac{C{\color{black}b_{d,n}}^2\beta^2\log(4d)}{s}
q^5(C^*q\beta {\color{black}b_{d,n}})^{q-4}.
\label{eq:order-q-common-main}
\end{align}
Lemmas~\ref{lem:nonterminal-common-base}
and~\ref{lem:terminal-order-Q} give
\begin{align*}
|\mathcal N_Q(s)|
&\leq
\frac{
C b_{d,n}^2\beta^2\log(4d)
}{s},
\\
|\mathcal T_Q(s)|
&\leq
C nQ^2(C^*Q\beta b_{d,n})^Q.
\end{align*}
Therefore, by
Eq.~\eqref{eq:NQ-TQ-decomposition},
\begin{align*}
\left|
\frac{d}{ds}
\mathbb E\pi_{\beta,a_0,a}(S_s)
\right|
&\leq
\frac{
C b_{d,n}^2\beta^2\log(4d)
}{s}
+
C nQ^2(C^*Q\beta b_{d,n})^Q.
\end{align*}
Therefore, integrating over $[\tau,1]$ gives 
\begin{align}
\int_\tau^1\left|\frac{d}{ds}\mathbb E\pi_{\beta,a_0,a}(S_s)\right|ds
&\leq
C{\color{black}b_{d,n}}^2\beta^2\log(4d)\log(1/\tau)
+
CnQ^2(C^*Q\beta {\color{black}b_{d,n}})^Q.
\label{eq:large-s-main-bound}
\end{align}
Combining Eq.~\eqref{eq:small-s-main-bound} and Eq.~\eqref{eq:large-s-main-bound} proves the proposition.
\end{proof}

\subsection{Gaussian replacement interpolation ($T\to G$)}
\label{subsec:T-to-G}

\begin{proposition}
\label{prop:T-to-G}
There exist universal constants $C,C^*>0$ such that, for every even integer $Q\geq8$ and every $\beta>0$ satisfying $C^*Q\beta {\color{black}b_{d,n}}\leq1/4$,
\begin{align*}
|\mathbb E\pi_{\beta,a_0,a}(T)-\mathbb E\pi_{\beta,a_0,a}(G)|
&\leq
C\left\{
{\color{black}b_{d,n}}\{\log(4d)\}^{3/2}
+
{\color{black}b_{d,n}}^2\beta^2\log(4d)
+
nQ^2(C^*Q\beta {\color{black}b_{d,n}})^Q
\right\}.
\end{align*}
\end{proposition}

\begin{proof}[Proof of Proposition~\ref{prop:T-to-G}]

For $0\leq s\leq1$, define
\begin{align*}
T_s
&=
\frac{1}{\sqrt2}G
+
2\sqrt{1-s}\sum_{i=1}^nD_iY_i
+
\sqrt{\frac{s}{2}}G_1,
\end{align*}
where $G,G_1\sim N(0,I_r)$ are independent of all other variables. Then $T_0\stackrel d=T$ and $T_1\stackrel d=G$.
For each $i$, define
$
R_{i,s}
=
\frac{1}{\sqrt2}G
+
2\sqrt{1-s}\sum_{\ell\neq i}D_\ell Y_\ell
+
\sqrt{\frac{s}{2}}G_1$.
Let $\mathcal F_{\mathrm{base}}
:=
\sigma\!\left(
G,G_1,(D_j,Y_j)_{1\leq j\leq n}
\right)$. On an enlarged probability space, choose the collection $\left\{
Y_i',\widetilde Y_i,D_i',\widetilde D_i
:
1\leq i\leq n
\right\}$ to be jointly independent and independent of
\(\mathcal F_{\mathrm{base}}\), with $\mathcal L(Y_i')
=
\mathcal L(\widetilde Y_i)
=
\mathcal L(Y_i)$ and $\mathcal L(D_i')
=
\mathcal L(\widetilde D_i)
=
\mathcal L(D_i)$ for every \(1\leq i\leq n\). Lemma~\ref{lem:Gaussian-replacement-generator} gives
\begin{align}
&\frac{d}{ds}\mathbb E\pi_{\beta,a_0,a}(T_s)
={}
-\frac{\sqrt{1-s}}{4}
\sum_{a_1,a_2,a_3=1}^r
\left\{\sum_{i=1}^n\mathbb E[Y_{ia_1}Y_{ia_2}Y_{ia_3}]\right\}
\mathbb E[\partial_{a_1a_2a_3}\pi_{\beta,a_0,a}(T_s)]
\nonumber\\
&+
\frac{1-s}{2}\sum_{i=1}^n\mathbb E\int_0^1(1-t)D^4\pi_{\beta,a_0,a}\left(R_{i,s}+2t\sqrt{1-s}D_iY_i\right)
\times[2D_iY_i,2D_iY_i,2D_i'Y_i',2D_i'Y_i']dt
\nonumber\\
&+
\frac{1-s}{4}\sum_{i=1}^n\mathbb E\int_0^1D^4\pi_{\beta,a_0,a}\left(R_{i,s}+2t\sqrt{1-s}D_iY_i\right)\times[2D_iY_i,2\widetilde D_i\widetilde Y_i,2\widetilde D_i\widetilde Y_i,2\widetilde D_i\widetilde Y_i]dt
\nonumber\\
&-
\frac{1-s}{4}\sum_{i=1}^n\mathbb E\int_0^1(1-t)^2D^4\pi_{\beta,a_0,a}\left(R_{i,s}+2t\sqrt{1-s}D_iY_i\right)
\times[(2D_iY_i)^{\otimes4}]dt.
\label{eq:main-Gaussian-generator}
\end{align}
We bound the third-order and fourth-order terms separately.

For the third-order term, let $\sigma_s^2=(1+s)/2$. 
Combining the Gaussian terms in $T_s$, applying third-order Gaussian integration by parts, and then applying the Gauss-Green formula gives
\begin{align}
&\sum_{a_1,a_2,a_3=1}^r
\left\{\sum_{i=1}^n\mathbb E[Y_{ia_1}Y_{ia_2}Y_{ia_3}]\right\}
\mathbb E[\partial_{a_1a_2a_3}\pi_{\beta,a_0,a}(T_s)]\\
&=
-\left(\frac{2}{(1+s)}\right)^{3/2}\mathbb E\sum_{j=1}^{m_P}\int_{F_j}\operatorname{tr}\{(xx^\top-I_r)B_j\}\varphi_r(x)d\mathcal H^{r-1}(x),
\label{eq:main-H3-surface}
\end{align}
where the polyhedron $P$ and its facets $F_j$ are defined in Appendix~\ref{sec:additional-notation}, and the weighted matrix $B_j$ is defined in Lemma~\ref{lem:skewness-operator} and satisfies 
\begin{align}
\max_{1\leq j\leq2d}\|B_j\|_{\mathrm{op}}
&\leq {\color{black}b_{d,n}}.
\label{eq:main-skewness-op}
\end{align}
Proposition~\ref{prop:matrix-surface}, Eq.~\eqref{eq:main-skewness-op}, and $(1+s)/2\geq1/2$ yield
\begin{align}
\left|
\frac{\sqrt{1-s}}{4}
\sum_{a_1,a_2,a_3=1}^r
\left\{\sum_{i=1}^n\mathbb E[Y_{ia_1}Y_{ia_2}Y_{ia_3}]\right\}
\mathbb E[\partial_{a_1a_2a_3}\pi_{\beta,a_0,a}(T_s)]
\right|
&\leq
C{\color{black}b_{d,n}}\{\log(4d)\}^{3/2}.
\label{eq:main-third-order-bound}
\end{align}
Lemma~\ref{lem:third-order-skewness} gives the complete justification of {\color{black}Eqs.}~\eqref{eq:main-H3-surface} {\color{black}and}~\eqref{eq:main-third-order-bound}.

For fourth-order terms, the independent Gaussian component of $R_{i,s}$ has covariance $(1+s)I_r/2\succeq I_r/2$. Hence the repeated-pair integration-by-parts estimate is uniform over $s\in[0,1]$. After $k$ ghost differences and a Taylor term of degree $h$, a nonterminal term has total derivative order $q=4+2k+h<Q$ and retains a repeated pair. Lemma~\ref{lem:T-to-G-fourth-order} {\color{black}bounds the absolute value of the contribution of the fourth-order terms by}
\begin{align*}
C{\color{black}b_{d,n}}^2\beta^2\log(4d)
\sum_{q=4}^{Q-1}q^5(C^*q\beta {\color{black}b_{d,n}})^{q-4}
+
CnQ^2(C^*Q\beta {\color{black}b_{d,n}})^Q
\leq
C{\color{black}b_{d,n}}^2\beta^2\log(4d)
+
CnQ^2(C^*Q\beta {\color{black}b_{d,n}})^Q.
\end{align*}
Combined with Eq.~\eqref{eq:main-third-order-bound}, we integrate Eq.~\eqref{eq:main-Gaussian-generator} over $s\in[0,1]$ and finish the proof.
\end{proof}

\section{Final hyperparameter selection}\label{sec:summary}

We assemble the results of Section~\ref{sec:preparation}, Section~\ref{sec:smooth-selector},  Section~\ref{sec:interpolation}, and choose the hyperparameters to complete the proof of Theorem~\ref{thm:subexponential-gaussian-approximation}.

Consider $b_{d,n}\leq1/4$ (if $b_{d,n}>1/4$, the desired conclusion follows from the trivial bound after increasing the constant), and we choose \(\delta=b_{d,n}\) in Proposition~\ref{prop:probability-comparison-through-selectors}.
Propositions~\ref{prop:S-to-T} and~\ref{prop:T-to-G},
Lemma~\ref{lem:conditional-nazarov}, 
Lemma~\ref{lem:preprocessing-transfer-main}, and a further quantile-transfer result in Section~\ref{appe:sec:final} then give us the upper bound
\begin{align}
\sup_{t\in\mathbb R}
\left|
\mathbb P\{T_n(a)\leq t\}
-
\mathbb P\{T_n^G(a)\leq t\}
\right|
\le
C_b\Bigg\{
&b_{d,n}
+
\frac{
\{1+\log_+(16d/b_{d,n})\}\sqrt{\log(4d)}
}{\beta}
+
nb_{d,n}^4\beta^4\tau
\nonumber\\
&+
b_{d,n}^2\beta^2\log(4d)\log(1/\tau)
+
b_{d,n}\{\log(4d)\}^{3/2}
\nonumber\\
&+
nQ^2({\color{black}C^*}Q\beta b_{d,n})^Q
+
n^{-20}
+
n^{-4}
\Bigg\}.
{\color{black}\label{eq:overall-bound-original-parameters-main}}
\end{align}
We next select hyperparameters
\begin{align}
\beta
&=
b_{d,n}^{-2/3}\{\log(4dn)\}^{-1/6},
\quad
\tau
=
\frac{
\{1+\log_+(16d{\color{black}n}/b_{d,n})\}\{\log(4dn)\}^{4/3}
}{
nb_{d,n}^{2/3}
},
{\color{black}\label{eq:corrected-parameter-choice-main}}
\end{align}
and \(Q=2\lceil C\log(2n)\rceil\), where \(C>0\) is a sufficiently large universal constant. 
These choices satisfy \(0<\tau<1\) and \({\color{black}C^*}Q\beta b_{d,n}\leq1/16\), and hence
$
nQ^2({\color{black}C^*}Q\beta b_{d,n})^Q
\leq
nQ^2 16^{-Q}
\leq
n^{-4}$.
Substituting these choices into
Eq.~{\color{black}\eqref{eq:overall-bound-original-parameters-main}} yields the bound in
Theorem~\ref{thm:subexponential-gaussian-approximation}. More details of the parameter selection are given in Section~{\color{black}\ref{appe:sec:final}} of the appendix.

\section{Discussion}
\label{sec:discussion}

There are two main directions for future work. First, the present result only provides an upper bound. Establishing the optimality of this rate would require a matching lower bound construction.
Second, it is of interest to extend the result to bootstrap approximation and improve the rate $n^{-1/4}$ {\color{black}of \citet{CCKK2022}}. Such an extension introduces the additional difficulty of a random calibration functional constructed from the same sample: the observed data determine both the test statistic and the bootstrap critical value.

% \section*{Acknowledgments}

% The initial proof attempt was generated by ChatGPT 5.6 Pro (OpenAI) and subsequently rewritten by Weihan Zhang and Zijun Gao.

\bibliographystyle{plainnat}
\bibliography{ref}

\appendix

\section{Additional notation}
\label{sec:additional-notation}

Let $d_{\mathrm{TV}}\{\mathbb P,\mathbb Q\}$ denote the total variation difference between $\mathbb P$ and $\mathbb Q$.
Here \(C_b^4(\mathbb R^r)\) denotes the class of functions
\(f:\mathbb R^r\to\mathbb R\) that are four times continuously
Fr\'echet differentiable and satisfy
\begin{align*}
\max_{0\leq k\leq4}
\sup_{x\in\mathbb R^r}
\|D^kf(x)\|_{\mathrm{op}}
&<
\infty,
\end{align*}
where \(D^0f=f\) and, for \(k\geq1\),
\begin{align*}
\|D^kf(x)\|_{\mathrm{op}}
&=
\sup_{\|z_1\|_2,\ldots,\|z_k\|_2\leq1}
\left|D^kf(x)[z_1,\ldots,z_k]\right|.
\end{align*}

For $a=(a_1,\ldots,a_{2d})\in\mathbb R^{2d}$, define
\begin{align*}
M_a(y)
=
\max_{1\leq j\leq2d}\{\langle u_j,y\rangle-a_j\},
\quad
P_a
=
\{y\in\mathbb R^r:M_a(y)\leq0\}.
\end{align*}
For the preprocessing step, write
\begin{align*}
\mu_n
&=
\frac{1}{\sqrt n}\sum_{i=1}^n\mathbb E[X_{i,A_i}].
\end{align*}
For an integer $k\geq0$, let $\mathfrak P_k$ be the set of partitions of $\{1,\ldots,k\}$, with $\mathfrak P_0$ containing the empty partition. For a sufficiently differentiable function $f$, $D^kf(x)[z_1,\ldots,z_k]$ denotes the $k$th Fr\'echet derivative evaluated in the displayed directions, $z^{\otimes k}$ is the $k$-fold tensor product, and $\Delta f=\operatorname{tr}(\nabla^2f)$. Let
\begin{align*}
H_2(x)
=
xx^\top-I_r,
\quad
H_{3,a_1a_2a_3}(x)
=
x_{a_1}x_{a_2}x_{a_3}-\delta_{a_1a_2}x_{a_3}-\delta_{a_1a_3}x_{a_2}-\delta_{a_2a_3}x_{a_1},
\end{align*}
where $\delta_{ab}=\mathbf1\{a=b\}$. 

\begin{definition}
\label{def:polyhedral-surface-configuration}
Let \(r,m_P\) be positive integers. A polyhedral surface
configuration consists of unit vectors
\(u_1,\ldots,u_{m_P}\in\mathbb R^r\) and thresholds
\(a_1,\ldots,a_{m_P}\in\mathbb R\) such that
\[
P
=
\bigcap_{j=1}^{m_P}
\left\{
x\in\mathbb R^r:
\langle u_j,x\rangle\leq a_j
\right\}
\]
is nonempty, proper, and full-dimensional, and the displayed
halfspace representation is irredundant.
\end{definition}

The following notation is used for every polyhedral surface
configuration throughout Appendix~\ref{sec:surface-appendix}. Its
facets are
\[
F_j
=
P\cap
\left\{
x\in\mathbb R^r:
\langle u_j,x\rangle=a_j
\right\},
\qquad j=1,\ldots,m_P.
\] For each $j$, let $E_j=u_j^\perp$ and
\begin{align*}
Q_j
=
\{y\in E_j:a_ju_j+y\in F_j\},
\quad
p_j
=
\gamma_{E_j}(Q_j),
\quad
s_j
=
\varphi_1(a_j)p_j,
\end{align*}
and let $\mu_j$ and $K_j$ be the conditional mean and covariance of a standard Gaussian on $E_j$ given $Q_j$. Put $\Delta_j=\operatorname{tr}(I_{E_j}-K_j)$. Relative to $\operatorname{span}(u_j)\oplus E_j$, write each symmetric matrix weight as
\begin{align*}
B_j
&=
\begin{pmatrix}
\alpha_j&v_j^\top\\
v_j&C_j
\end{pmatrix}.
\end{align*}

For a codimension-two ridge $F_{jk}=F_j\cap F_k$, let $\alpha_{jk}\in(0,\pi)$ be the angle between $u_j$ and $u_k$, let $N_{jk}=\operatorname{span}(u_j,u_k)$ and $H_{jk}=N_{jk}^\perp$, and let $z_{jk}\in N_{jk}$ satisfy $\langle u_j,z_{jk}\rangle=a_j$ and $\langle u_k,z_{jk}\rangle=a_k$. Define
\begin{align*}
C_{jk}
=
\operatorname{cone}(u_j,u_k),
\quad
b_{jk}
=
\frac{u_j+u_k}{2\cos(\alpha_{jk}/2)},
\quad
\rho_{jk}
=
\int_{F_{jk}}\varphi_r(x)d\mathcal H^{r-2}(x),
\end{align*}
\begin{align*}
\mathcal N_{jk}
=
\operatorname{relint}(F_{jk})+\operatorname{relint}(C_{jk}),
\quad
q_{jk}
=
\gamma_r(\mathcal N_{jk}),
\quad\quad
T_{jk}
=
\{t\in H_{jk}:z_{jk}+t\in\operatorname{relint}(F_{jk})\},
\end{align*}
and $p_{jk}=\gamma_{H_{jk}}(T_{jk})$. For the ridge induced inside the facet section $Q_j$, define
\begin{align*}
n_{k\mid j}
=
\frac{u_k-\langle u_j,u_k\rangle u_j}{\sqrt{1-\langle u_j,u_k\rangle^2}},
\quad
d_{k\mid j}
=
\frac{a_k-\langle u_j,u_k\rangle a_j}{\sqrt{1-\langle u_j,u_k\rangle^2}},
\end{align*}
and let $r_{jk}^{(j)}$ be the Gaussian surface mass of $\{y\in Q_j:\langle n_{k\mid j},y\rangle=d_{k\mid j}\}$ in $E_j$. These quantities are defined only for adjacent facets, and every ridge sum below ranges over unordered adjacent pairs $\{j,k\}$.

\section{Literature}\label{appe:sec:literature}

The modern theory of high-dimensional Gaussian approximation over
hyperrectangles was initiated by \citet{CCK2013} and further developed by
\citet{CCK2017}; see \citet{CCKK2023Review} for a review.  For general
independent arrays with unrestricted covariance, \citet{CCK2017} obtained a
representative bound of order
\[
    \left\{
        \frac{B_n^2\log^7(dn)}{n}
    \right\}^{1/6},
\]
and \citet{Koike2021} reduced the logarithmic power from seven to five.
Related bootstrap improvements were obtained by \citet{DengZhang2020}.
The sharpest previous unrestricted-covariance benchmark is
\citet{CCKK2022}, who established
\[
    \left\{
        \frac{B_n^2\log^5(dn)}{n}
    \right\}^{1/4}
\]
under coordinatewise subexponential tails, marginal variance lower bounds,
and fourth-moment control.  These results allow the covariance matrix to be
singular and rely on randomized Lindeberg interpolation, smooth
approximations of maxima, and Gaussian anti-concentration.

Faster rates are available under additional covariance or distributional
structure.  Under covariance nondegeneracy,
\citet{FangKoike2021}, \citet{KuchibhotlaRinaldo2020}, and
\citet{Lopes2022} obtained cubic-root or near-\(n^{-1/2}\) rates.
For bounded summands and strongly nondegenerate covariance,
\citet{CCK2023} proved the nearly sharp bound
\[
    \frac{B_n(\log d)^{3/2}\log n}{\sqrt n},
\]
together with a matching lower bound up to the factor \(\log n\).
Global singularity alone does not rule out such rates:
\citet{FangKoikeLiuZhao2023} obtained a near-\(n^{-1/2}\) result when every
three-coordinate subsystem is uniformly nondegenerate.  Separately,
\citet{FangKoike2024} derived a near-\(n^{-1/2}\) rate for i.i.d.\
log-concave summands without assuming covariance nondegeneracy.  These
results indicate that the main obstruction is not global rank deficiency
alone, but the combination of local degeneracy and weak distributional
structure.

Theorem~\ref{thm:subexponential-gaussian-approximation} returns to the general
independent array setting and permits arbitrary, possibly singular,
covariance.  Its bound is of smaller order than the quarter-power bound of
\citet{CCKK2022} whenever
\[
    \frac{B_n^2\{\log(2dn)\}^{13}}{n}\longrightarrow0.
\]
In particular, for fixed \(B_n\) and polynomial dimension, the previous
unrestricted-covariance rate is \(n^{-1/4}\), whereas the present theorem
gives \(n^{-1/3}\).  The statement in \citet{CCK2023} that the
quarter-power bound is conjectured to be near-optimal was informal and did
not specify the minimax quantifiers.  Under the fixed-envelope,
polynomial-dimensional formulation adopted in this paper, the present upper
bound rules out that conjecture.  

Methodologically, our argument differs from both the iterative randomized
Lindeberg method of \citet{CCKK2022} and the implicit-smoothing or mixed-smoothing
methods used under covariance nondegeneracy. We first restrict the problem
to the covariance support, then introduce a Gaussian-divisible proxy matching
the original sum through third moments.  A ghost-completion argument and
repeated-direction Gaussian integration by parts control the resulting
fourth-order terms, while a rank-free matrix-weighted Gaussian surface bound
controls the remaining third-order term.  Higher-order Edgeworth and
bootstrap expansions, such as those of \citet{FangLiu2025} and
\citet{Koike2026}, retain rather than bound this third-order correction and
are therefore complementary to our results.

% \paragraph{Related work.}
% High-dimensional Gaussian approximation over rectangles has been developed for three broad regimes. For general arrays with unrestricted covariance, the representative sample-size exponents progressed from $1/6$ in \citet{CCK2017} and \citet{Koike2021} to $1/4$ in \citet{CCKK2022}. Under covariance nondegeneracy, cubic-root and near-$n^{-1/2}$ rates are available; see \citet{FangKoike2021}, \citet{KuchibhotlaRinaldo2020}, \citet{Lopes2022}, \citet{CCK2023}, and \citet{FangKoikeLiuZhao2023}. For i.i.d. log-concave summands, \citet{FangKoike2024} obtain a near-$n^{-1/2}$ rate without requiring covariance nondegeneracy. Table~\ref{tab:literature-summary} summarizes these comparisons, suppressing logarithmic factors.

\section{Example}

\begin{example}[Counterexample without Assumption~\ref{assu:marginal-variance-lower-bound}]
\label{exam:counter.variance.lower.bound}
Let \(d=1\), \(a=0\), \(X_1=\varepsilon_n\epsilon\), and \(X_i=0\) for \(i\geq2\), where \(\epsilon\) is a Rademacher random variable and \(\varepsilon_n\downarrow0\). Taking \(B_n=\varepsilon_n/\log 2\), Assumption~\ref{assu:sub.exp} holds, but \(n^{-1}\sum_{i=1}^n\mathbb E[X_i^2]=\varepsilon_n^2/n\). For \(\alpha=1/4\), the left-hand side of \eqref{eq:subexponential-gaussian-critical-value-bound} equals \(1/4\), whereas the rate on the right-hand side converges to zero.
\end{example}

\section{Supporting results for Section~\ref{sec:preparation}}
\label{sec:preparation-appendix}

\begin{lemma}
\label{lem:positive-diagonal-invariance}
Let $V,W\in\mathbb R^d$ and let $D$ be a positive diagonal matrix. Then
\begin{align*}
\sup_{A\in\mathcal R^d}|\mathbb P(V\in A)-\mathbb P(W\in A)|
&=
\sup_{A\in\mathcal R^d}|\mathbb P(DV\in A)-\mathbb P(DW\in A)|.
\end{align*}
\end{lemma}

\begin{proof}
For every $A\in\mathcal R^d$, $D^{-1}A\in\mathcal R^d$ and
\begin{align*}
\mathbb P(DV\in A)-\mathbb P(DW\in A)
&=
\mathbb P(V\in D^{-1}A)-\mathbb P(W\in D^{-1}A).
\end{align*}
Since $A\mapsto D^{-1}A$ is a bijection of $\mathcal R^d$, taking the two suprema proves the result.
\end{proof}

\begin{lemma}
\label{lem:closed-rectangles}
For any probability measures \(\mu,\nu\) on \(\mathbb R^d\),
\begin{align*}
\sup_{A\in\mathcal R^d}
|\mu(A)-\nu(A)|
&=
\sup_{\substack{
a^-,a^+\in\mathbb R^d\\
a^-\leq a^+\ {\rm coordinatewise}
}}
\left|
\mu\left(\prod_{j=1}^d[a_j^-,a_j^+]\right)
-
\nu\left(\prod_{j=1}^d[a_j^-,a_j^+]\right)
\right|.
\end{align*}
Thus the supremum may be restricted to nonempty closed rectangles
whose endpoints are all finite.
\end{lemma}

\begin{proof}
If \(A\in\mathcal R^d\) is empty, then $|\mu(A)-\nu(A)|=0$, so empty rectangles do not affect the supremum. Now let $A=\prod_{j=1}^d I_j\in\mathcal R^d$ be nonempty. Then every \(I_j\) is a nonempty interval. For each
\(j\), choose an increasing sequence of nonempty bounded closed
intervals such that $I_{j,m}\subseteq I_j$ and $I_{j,m}\uparrow I_j$. Such a sequence is obtained by moving every excluded finite endpoint inward and truncating every infinite endpoint; finitely many initial empty terms, if any, are discarded. Set $A_m:=\prod_{j=1}^d I_{j,m}$. Then every \(A_m\) is a nonempty closed rectangle with finite
endpoints, and, since \(d\) is finite, $A_m\uparrow A$. Continuity from below gives $\mu(A_m)\longrightarrow\mu(A)$ and $\nu(A_m)\longrightarrow\nu(A)$, and therefore $|\mu(A_m)-\nu(A_m)|
\longrightarrow
|\mu(A)-\nu(A)|$. Hence the supremum over all rectangles is bounded by the supremum over nonempty closed rectangles with finite endpoints. The reverse inequality is immediate because every such rectangle belongs to \(\mathcal R^d\).
\end{proof}

\begin{lemma}
\label{lem:product-conditioning}
Let \(X_1,\ldots,X_n\) be independent random elements, where
\(X_i\) takes values in a measurable space
\((\mathcal X_i,\mathcal A_i)\).
For each \(i\), let \(E_i\in\mathcal A_i\) satisfy $\mathbb P(X_i\in E_i)>0$, and put $A_i:=\{X_i\in E_i\}$. Let \(X_{1,E_1},\ldots,X_{n,E_n}\) be independent random elements
satisfying $\mathcal L(X_{i,E_i})
=
\mathcal L(X_i\mid X_i\in E_i)$. Then
\begin{align*}
d_{\mathrm{TV}}
\left\{
\mathcal L(X_1,\ldots,X_n),
\mathcal L(X_{1,E_1},\ldots,X_{n,E_n})
\right\}
&=
1-\prod_{i=1}^n\mathbb P(X_i\in E_i)
\\
&\le
\sum_{i=1}^n\mathbb P(X_i\notin E_i).
\end{align*}
The same upper bound holds after applying any measurable map to the
product vector.
\end{lemma}

\begin{proof}
Since \(A_i\in\sigma(X_i)\) and the random elements
\(X_1,\ldots,X_n\) are independent, the events
\(A_1,\ldots,A_n\) are independent. Hence
\[
\mathbb P\left(\bigcap_{i=1}^nA_i\right)
=
\prod_{i=1}^n\mathbb P(A_i).
\]
Moreover, the product of the componentwise conditional laws is the
law of \((X_1,\ldots,X_n)\) conditional on
\(\bigcap_{i=1}^nA_i\). For every probability measure \(P\) and
measurable event \(A\) with \(P(A)>0\),
\[
d_{\mathrm{TV}}\{P,P(\cdot\mid A)\}=P(A^c).
\]
Therefore,
\[
d_{\mathrm{TV}}
\left\{
\mathcal L(X_1,\ldots,X_n),
\mathcal L(X_{1,E_1},\ldots,X_{n,E_n})
\right\}
=
1-\prod_{i=1}^n\mathbb P(A_i).
\]
The union bound gives
\[
1-\prod_{i=1}^n\mathbb P(A_i)
\le
\sum_{i=1}^n\mathbb P(A_i^c).
\]
Finally, total variation contracts under measurable pushforwards.
\end{proof}

\begin{lemma}
\label{lem:subexponential-tail-budgets}
Let \(B>0\). If $\mathbb E\exp(|X_j|/B)\leq2$ for $j=1,\ldots,d$, then, for every $t\geq0$,
\begin{align*}
\mathbb P(\|X\|_\infty>t)
&\leq
2de^{-t/B},\\
\mathbb E[\|X\|_\infty\mathbf1\{\|X\|_\infty>t\}]
&\leq
2d(t+B)e^{-t/B},\\
\mathbb E[\|X\|_\infty^2\mathbf1\{\|X\|_\infty>t\}]
&\leq
2d(t^2+2tB+2B^2)e^{-t/B}.
\end{align*}
\end{lemma}

\begin{proof}
Markov's inequality and a union bound give
\begin{align*}
\mathbb P(\|X\|_\infty>t)
&\leq
\sum_{j=1}^d\mathbb P(|X_j|>t)
\leq
2de^{-t/B}.
\end{align*}
For a nonnegative random variable $V$, we have
\begin{align*}
\mathbb E[V\mathbf1\{V>t\}]
&=
t\mathbb P(V>t)+\int_t^\infty \mathbb P(V>s)ds,\\
\mathbb E[V^2\mathbf1\{V>t\}]
&=
t^2\mathbb P(V>t)+2\int_t^\infty s\mathbb P(V>s)ds.
\end{align*}
Substituting the tail bound and evaluating the integrals proves the two moment inequalities.
\end{proof}

\begin{lemma}
\label{lem:conditional-centering}
There exists a universal constant \(C>0\) such that the following
holds. Let $X\in\mathbb R^d$ be centered, let $A=\{\|X\|_\infty\leq t\}$, and suppose $\mathbb P(A^c)\leq1/2$. Let $X_A\sim\mathbb P(X\mid A)$. If $\max_j\mathbb E[X_j^2]\leq M^2$, then
\begin{align*}
\|\mathbb E[X_A]\|_\infty
&\leq
2\mathbb E[\|X\|_\infty\mathbf1_{A^c}],
\\
\max_{j,k}\left|\operatorname{Cov}(X_A)_{jk}-\mathbb E[X_jX_k]\right|
&\leq
C\left\{
\mathbb E[\|X\|_\infty^2\mathbf1_{A^c}]
+
\mathbb P(A^c)M^2
+
\mathbb E[\|X\|_\infty\mathbf1_{A^c}]^2
\right\}.
\end{align*}
\end{lemma}

\begin{proof}
Since $\mathbb E[X]=0$, $\mathbb E[X_A]
=
-\frac{\mathbb E[X\mathbf1_{A^c}]}{\mathbb P(A)}$,
so $\mathbb P(A)\geq1/2$ gives the first inequality. 

Moreover,
\begin{align*}
\mathbb E[(X_A)_j(X_A)_k]-\mathbb E[X_jX_k]
&=
\frac{\mathbb P(A^c)\mathbb E[X_jX_k]-\mathbb E[X_jX_k\mathbf1_{A^c}]}{\mathbb P(A)}.
\end{align*}
Using $|\mathbb E[X_jX_k]|\leq M^2$, $|X_jX_k|\leq\|X\|_\infty^2$, and
$
\operatorname{Cov}(X_A)_{jk}
=
\mathbb E[(X_A)_j(X_A)_k]-\mathbb E[(X_A)_j]\mathbb E[(X_A)_k]$
proves the second inequality.
\end{proof}

\begin{lemma}
\label{lem:Gaussian-transfer}
For every \(c>0\), there exists a constant \(C_c\in(0,\infty)\)
such that the following holds. Let $G_1\sim N(0,\Sigma_1)$ and $G_2\sim N(0,\Sigma_2)$ in $\mathbb R^d$, and suppose all diagonal entries of $\Sigma_1$ and $\Sigma_2$ are at least $c>0$. Then, for every $\mu\in\mathbb R^d$,
\begin{align*}
\sup_{A\in\mathcal R^d}|\mathbb P(G_1+\mu\in A)-\mathbb P(G_1\in A)|
&\leq
C_c\|\mu\|_\infty\sqrt{\log(2d)},
\\
\sup_{A\in\mathcal R^d}|\mathbb P(G_1\in A)-\mathbb P(G_2\in A)|
&\leq
C_c\min\left\{1,\left[\|\Sigma_1-\Sigma_2\|_\infty\{\log(2d)\}^2\right]^{1/2}\right\}.
\end{align*}
\end{lemma}

\begin{proof}
The shift bound follows by applying Nazarov anti-concentration to the doubled vector $(G_1,-G_1)$; see \citet{CCK2015}. The covariance bound is the arbitrary-covariance Gaussian comparison inequality of \citet{CCKK2022}.
\end{proof}

\begin{lemma}
\label{lem:critical-value-transfer}
Let \(V\) and \(V^G\) be real-valued random variables, and define $F(t)=\mathbb P(V\le t)$ and $F^G(t)=\mathbb P(V^G\le t)$. Suppose \(F^G\) is continuous and $\sup_{t\in\mathbb R}|F(t)-F^G(t)|\le\varepsilon$. For \(\alpha\in(0,1)\), let
\[
q_{1-\alpha}^G
=
\inf\{t:F^G(t)\ge1-\alpha\}.
\]
Then
\[
\left|
\mathbb P\{V>q_{1-\alpha}^G\}-\alpha
\right|
\le\varepsilon.
\]
\end{lemma}

\begin{proof}
Continuity of \(F^G\), together with its limits at
\(\pm\infty\), gives $F^G(q_{1-\alpha}^G)=1-\alpha$. Therefore,
\begin{align*}
\left|
\mathbb P\{V>q_{1-\alpha}^G\}-\alpha
\right|
&=
\left|
F(q_{1-\alpha}^G)-(1-\alpha)
\right|
\\
&=
\left|
F(q_{1-\alpha}^G)-F^G(q_{1-\alpha}^G)
\right|
\le\varepsilon.
\end{align*}
\end{proof}

\begin{proof}[Proof of
Lemma~\ref{lem:preprocessing-transfer-main}]
Work throughout under
Eq.~\eqref{eq:preprocessing-regime}.
Set $t=32B_n\log(2dn)$. Lemma~\ref{lem:subexponential-tail-budgets} gives, for every $i$,
\begin{align*}
\mathbb P(A_i^c)
&\leq
2d\exp\{-32\log(2dn)\}
=
2d(2dn)^{-32},
\end{align*}
and hence
\begin{align*}
\sum_{i=1}^n\mathbb P(A_i^c)
&\leq
2nd(2dn)^{-32}
\leq
n^{-20}.
\end{align*}
The nontrivial regime implies
\begin{align*}
\left(\frac{B_n^2}{n}\right)^{1/3}\{\log(2dn)\}^{7/3}
\leq C,
\quad
\frac{B_n}{\sqrt n}
\leq
C\{\log(2dn)\}^{-7/2}.
\end{align*}
Since $\mathbb E[X_i]=0$ and $\mathbb P(A_i)\geq1/2$,
\begin{align*}
\left\|\mathbb E[X_{i,A_i}]\right\|_\infty
&\leq
2\mathbb E[\|X_i\|_\infty\mathbf1_{A_i^c}]
\leq
4d\{t+B_n\}(2dn)^{-32}.
\end{align*}
Therefore,
\begin{align*}
\left\|\frac{1}{\sqrt n}\sum_{i=1}^n\mathbb E[X_{i,A_i}]\right\|_\infty
&\leq
4\sqrt n\,d\{t+B_n\}(2dn)^{-32}
\leq
n^{-10}.
\end{align*}
Assumption~\ref{assu:sub.exp} also gives $\max_j\mathbb E[X_{ij}^2]\leq CB_n^2$. Lemmas~\ref{lem:conditional-centering} and~\ref{lem:subexponential-tail-budgets} yield
\begin{align*}
\|\widetilde\Sigma_n-\Sigma_n\|_\infty
\leq
\frac{C}{n}\sum_{i=1}^n\Big[
2d(t^2+2tB_n+2B_n^2)(2dn)^{-32}
+
2dB_n^2(2dn)^{-32}
+
\{2d(t+B_n)(2dn)^{-32}\}^2
\Big]
\leq
n^{-10}.
\end{align*}
This proves Eq.~\eqref{eq:preprocessing-errors-main}. In the preceding calculation, $\widetilde\Sigma_n$ denotes the covariance before the final marginal variance standardization. It also gives
\begin{align*}
(\widetilde\Sigma_n)_{jj}
&\geq
(\Sigma_n)_{jj}-n^{-10}
\geq
b^2/2
\end{align*}
for all sufficiently large $n$.

Let $X_{1,A_1},\ldots,X_{n,A_n}$ be coupled through the product conditioning law. Lemma~\ref{lem:product-conditioning} gives
\begin{align*}
d_{\mathrm{TV}}\left\{\mathbb P\left(\frac{1}{\sqrt n}\sum_{i=1}^nX_i\right),\mathbb P\left(\frac{1}{\sqrt n}\sum_{i=1}^nX_{i,A_i}\right)\right\}
&\leq
n^{-20}.
\end{align*}
Since
\begin{align*}
\frac{1}{\sqrt n}\sum_{i=1}^nX_{i,A_i}
&=
\frac{1}{\sqrt n}\sum_{i=1}^n\widetilde X_i
+
\frac{1}{\sqrt n}\sum_{i=1}^n\mathbb E[X_{i,A_i}],
\end{align*}
translation invariance and the triangle inequality give
\begin{align*}
\sup_{A\in\mathcal R^d}|\mathbb P(W\in A)-\mathbb P(Z\in A)|
\leq{}&
n^{-20}
+
\sup_{A\in\mathcal R^d}\left|
\mathbb P\left(\frac{1}{\sqrt n}\sum_{i=1}^n\widetilde X_i\in A\right)
-\mathbb P\{N(0,\widetilde\Sigma_n)\in A\}
\right|
\nonumber\\
&+
\sup_{A\in\mathcal R^d}\left|
\mathbb P\{N(\mu_n,\widetilde\Sigma_n)\in A\}
-\mathbb P\{N(0,\Sigma_n)\in A\}
\right|.
\end{align*}
By Lemma~\ref{lem:Gaussian-transfer}, the final term is bounded by
\begin{align*}
C_{b}\left[
n^{-10}\sqrt{\log(2d)}
+
\left\{n^{-10}\{\log(2d)\}^2\right\}^{1/2}
\right]
&\leq
C_b n^{-4}.
\end{align*}
Finally, Lemma~\ref{lem:positive-diagonal-invariance} identifies the middle rectangle distance with the one based on the final marginally standardized variables. This proves Lemma~\ref{lem:preprocessing-transfer-main}.
\end{proof}

\section{Supporting results for Section~\ref{sec:smooth-selector}}
\label{sec:selector-appendix}

\begin{lemma}
\label{lem:one-sided-selector}
Fix $0<\delta\leq1/4$, let $h=\log(8d/\delta)$, and put $a_0^+=h/\beta$ and $a_0^-=-h/\beta$. Then, for every $y\in\mathbb R^r$,
\begin{align*}
(1-\delta)\mathbf1\{M_a(y)\leq0\}
&\leq
\pi_{\beta,a_0^+,a}(y)
\leq
\mathbf1\left\{M_a(y)\leq\frac{2h}{\beta}\right\}+\delta,
\\
(1-\delta)\mathbf1\left\{M_a(y)\leq-\frac{2h}{\beta}\right\}
&\leq
\pi_{\beta,a_0^-,a}(y)
\leq
\mathbf1\{M_a(y)\leq0\}+\delta.
\end{align*}
\end{lemma}

\begin{proof}
If $M_a(y)\leq0$, then
\begin{align*}
\sum_{j=1}^{2d}\exp\{\beta(\langle u_j,y\rangle-a_j)\}
\leq
2d,
\quad 
\exp(\beta a_0^+)
=
\frac{8d}{\delta},
\end{align*}
which gives us
\begin{align*}
\pi_{\beta,a_0^+,a}(y)
&\geq
\frac{8d/\delta}{2d+8d/\delta}
=
\frac{1}{1+\delta/4}
\geq
1-\delta.
\end{align*}
If $M_a(y)>2h/\beta$, one real-score exponential exceeds $e^{2h}$, and
\begin{align*}
\pi_{\beta,a_0^+,a}(y)
&\leq
e^{-h}
=
\frac{\delta}{8d}
\leq
\delta.
\end{align*}
This proves the first pair of inequalities. If $M_a(y)\leq-2h/\beta$, then
\begin{align*}
\frac{\sum_{j=1}^{2d}\exp\{\beta(\langle u_j,y\rangle-a_j)\}}{\exp(\beta a_0^-)}
&\leq
2de^{-h}
=
\delta/4,
\end{align*}
which gives $\pi_{\beta,a_0^-,a}(y)\geq(1+\delta/4)^{-1}\geq1-\delta$. If $M_a(y)>0$, one real-score exponential exceeds one, and
\begin{align*}
\pi_{\beta,a_0^-,a}(y)
&\leq
e^{-h}
\leq
\delta.
\end{align*}
This proves the second pair.
\end{proof}

\begin{proof}[Proof of Proposition~\ref{prop:probability-comparison-through-selectors}]
Put $w=2h/\beta$. Lemma~\ref{lem:one-sided-selector} and Eq.~\eqref{eq:selector-comparison-hypothesis} give
\begin{align*}
(1-\delta)\mathbb P\{M_a(S)\leq0\}
&\leq
\mathbb E\pi_{\beta,a_0^+,a}(S)
\leq
\mathbb E\pi_{\beta,a_0^+,a}(G)+\Delta
\leq
\mathbb P\{M_a(G)\leq w\}+\delta+\Delta.
\end{align*}
Since $\delta\leq1/4$,
\begin{align*}
\mathbb P\{M_a(S)\leq0\}-\mathbb P\{M_a(G)\leq0\}
&\leq
\mathbb P\{0<M_a(G)\leq w\}+C(\delta+\Delta).
\end{align*}
For the reverse difference, replace every real threshold $a_j$ by $a_j-w$, so the resulting maximum equals $M_a+w$. Interchanging $S$ and $G$ in Eq.~\eqref{eq:selector-comparison-hypothesis} gives
\begin{align*}
(1-\delta)\mathbb P\{M_a(G)\leq-w\}
&\leq
\mathbb P\{M_a(S)\leq0\}+\delta+\Delta,
\end{align*}
and hence
\begin{align*}
\mathbb P\{M_a(G)\leq0\}-\mathbb P\{M_a(S)\leq0\}
&\leq
\mathbb P\{-w<M_a(G)\leq0\}+C(\delta+\Delta).
\end{align*}
Both shells are bounded by the supremum in the proposition. Taking the supremum over $a$ proves the result.
\end{proof}

\begin{lemma}
\label{lem:conditional-nazarov}
There exists a universal constant $C>0$ such that the following holds. Let $H\in\mathbb R^r$, let $G\sim N(0,I_r)$ be independent of $H$, and let $\sigma\geq\sigma_0>0$. Then, for every $\varepsilon\geq0$,
\begin{align*}
\sup_{\substack{a\in\mathbb R^{2d}\\t\in\mathbb R}}
\mathbb P\{t<M_a(H+\sigma G)\leq t+\varepsilon\}
&\leq
C\sigma_0^{-1}\varepsilon\sqrt{\log(4d)}.
\end{align*}
\end{lemma}

\begin{proof}
Condition on $H$. The Gaussian vector $(\langle u_j,G\rangle)_{j\leq2d}$ has unit marginal variances, and
\begin{align*}
\{M_a(H+\sigma G)\leq t\}
&=
\bigcap_{j=1}^{2d}
\left\{
\langle u_j,G\rangle
\leq
\frac{t+a_j-\langle u_j,H\rangle}{\sigma}
\right\}.
\end{align*}
Nazarov's inequality \citep{CCK2015} therefore gives
\begin{align*}
\mathbb P\{t<M_a(H+\sigma G)\leq t+\varepsilon\mid H\}
&\leq
C\frac{\varepsilon}{\sigma}\sqrt{\log(4d)}
\leq
C\sigma_0^{-1}\varepsilon\sqrt{\log(4d)}.
\end{align*}
Averaging over $H$ proves the claim.
\end{proof}

\begin{lemma}
\label{lem:convex-conditioning}
If $G\sim N(0,I_r)$ and $C\subseteq\mathbb R^r$ is convex and Borel with $\mathbb P(G\in C)>0$, then
\begin{align*}
\operatorname{Cov}(G\mid G\in C)
&\preceq
I_r.
\end{align*}
\end{lemma}

\begin{proof}
The conditional density is proportional to $\exp\{-\|x\|_2^2/2-\iota_C(x)\}$, where $\iota_C$ is convex. The Brascamp-Lieb covariance inequality \citep{brascamp1976extensions} gives, for every $v\in\mathbb R^r$,
\begin{align*}
\operatorname{Var}(v^\top G\mid G\in C)
&\leq
\|v\|_2^2,
\end{align*}
which is equivalent to the matrix inequality.
\end{proof}

\begin{lemma}
\label{lem:convex-cell-contraction}
Let $C_1,\ldots,C_N$ be convex Borel cells partitioning $\mathbb R^r$ modulo a standard-Gaussian-null set, and let $A_1,\ldots,A_N$ be symmetric matrices satisfying $\max_{\ell\leq N}\|A_\ell\|_{\mathrm{op}}\leq A$. Then
\begin{align*}
\left|
\sum_{\ell=1}^N
\mathbb E\left[
\{G^\top A_\ell G-\operatorname{tr}(A_\ell)\}\mathbf1\{G\in C_\ell\}
\right]
\right|
&\leq
4A\log N.
\end{align*}
\end{lemma}

\begin{proof}
Discard cells with zero Gaussian probability and write
\begin{align*}
p_\ell
=
\mathbb P(G\in C_\ell),
\quad
\mu_\ell
=
\mathbb E(G\mid G\in C_\ell),
\quad
K_\ell
=
\operatorname{Cov}(G\mid G\in C_\ell).
\end{align*}
Then
\begin{align*}
\mathbb E[(GG^\top-I_r)\mathbf1\{G\in C_\ell\}]
&=
p_\ell(K_\ell-I_r+\mu_\ell\mu_\ell^\top).
\end{align*}
Lemma~\ref{lem:convex-conditioning} gives $K_\ell\preceq I_r$, while the law of total covariance gives
\begin{align*}
\sum_{\ell=1}^Np_\ell\operatorname{tr}(I_r-K_\ell)
&=
\sum_{\ell=1}^Np_\ell\|\mu_\ell\|_2^2.
\end{align*}
The entropy variational inequality gives $\|\mu_\ell\|_2^2\leq2\log(1/p_\ell)$. Therefore, by nuclear/operator norm duality,
\begin{align*}
\left|
\sum_{\ell=1}^N
\mathbb E\left[
\{G^\top A_\ell G-\operatorname{tr}(A_\ell)\}\mathbf1\{G\in C_\ell\}
\right]
\right|
&\leq
4A\sum_{\ell=1}^Np_\ell\log(1/p_\ell)
\leq
4A\log N.
\end{align*}
\end{proof}

\begin{lemma}
\label{lem:selector-derivatives}
For every integer $k\geq1$,
\begin{align*}
D^k\pi_{\beta,a_0,a}(y)[z_1,\ldots,z_k]
&=
\beta^k\pi_{\beta,a_0,a}(y)
\sum_{\mathcal Q\in\mathfrak P_k}
(-1)^{|\mathcal Q|}|\mathcal Q|!
\prod_{B\in\mathcal Q}
\left\{
\sum_{j=1}^{2d}\pi_j(y)\prod_{\ell\in B}\langle u_j,z_\ell\rangle
\right\}.
\end{align*}
Consequently, for a universal constant $C_0>0$,
\begin{align*}
|D^k\pi_{\beta,a_0,a}(y)[z_1,\ldots,z_k]|
&\leq
C_0^k k!\beta^k
\prod_{\ell=1}^k\max_{j\leq2d}|\langle u_j,z_\ell\rangle|.
\end{align*}
\end{lemma}

\begin{proof}
Let $J$ be categorical with probabilities $(\pi_0(y),\ldots,\pi_{2d}(y))$ and set $u_0=0$. Differentiation under the exponential tilt gives
\begin{align*}
D^k\pi_{\beta,a_0,a}(y)[z_1,\ldots,z_k]
&=
\beta^k\operatorname{cum}\{\mathbf1(J=0),\langle u_J,z_1\rangle,\ldots,\langle u_J,z_k\rangle\}.
\end{align*}
In the joint-cumulant partition formula, every block containing both $\mathbf1(J=0)$ and a score has zero expectation because $u_0=0$. Hence the indicator forms a singleton block, and
\begin{align*}
\mathbb E\prod_{\ell\in B}\langle u_J,z_\ell\rangle
&=
\sum_{j=1}^{2d}\pi_j(y)\prod_{\ell\in B}\langle u_j,z_\ell\rangle,
\end{align*}
which gives the identity. Moreover,
\begin{align*}
\left|
\sum_{j=1}^{2d}\pi_j(y)\prod_{\ell\in B}\langle u_j,z_\ell\rangle
\right|
&\leq
\prod_{\ell\in B}\max_{j\leq2d}|\langle u_j,z_\ell\rangle|.
\end{align*}
Lemma~\ref{lem:ordered-Bell-bound} then gives the stated bound.
\end{proof}

\begin{lemma}
\label{lem:Gumbel-cells}
Let $\gamma_0,\ldots,\gamma_{2d}$ be independent standard Gumbel variables and use the smallest-index tie rule. Then
\begin{align*}
\mathbb P\left\{
\operatorname*{arg\,max}_{0\leq j\leq2d}
\left(s_j(y)+\frac{\gamma_j}{\beta}\right)=j
\right\}
&=
\pi_j(y).
\end{align*}
For fixed offsets, every winner cell is convex and Borel, the dummy cell is a closed polyhedron, and the cells partition $\mathbb R^r$ modulo Gaussian-null tie hyperplanes for almost every offset realization.
\end{lemma}

\begin{proof}
The probability identity is the Gumbel-max formula \citep{MaddisonTarlowMinka2014}. Under the smallest-index rule, label $j$ wins exactly on
\begin{align*}
\bigcap_{k<j}\left\{y:s_j(y)+\gamma_j/\beta>s_k(y)+\gamma_k/\beta\right\}
\cap
\bigcap_{k>j}\left\{y:s_j(y)+\gamma_j/\beta\geq s_k(y)+\gamma_k/\beta\right\}.
\end{align*}
This is an intersection of affine halfspaces and is therefore convex and Borel. Except on a probability-zero set of offsets, each tie set is empty or an affine hyperplane, so replacing strict inequalities by weak ones changes the cells only on a Gaussian-null set.
\end{proof}

\begin{lemma}
\label{lem:ordered-Bell-bound}
There exists a universal constant $C_0>0$ such that, for every $k\geq0$,
\begin{align*}
\sum_{\mathcal Q\in\mathfrak P_k}|\mathcal Q|!
&\leq
C_0^k k!.
\end{align*}
\end{lemma}

\begin{proof}
The ordered-Bell exponential generating function is
\begin{align*}
\sum_{k=0}^\infty
\left\{\sum_{\mathcal Q\in\mathfrak P_k}|\mathcal Q|!\right\}
\frac{z^k}{k!}
&=
\frac{1}{2-e^z},
\qquad |z|<\log2;
\end{align*}
see \citet[Exercise~20, p.~228]{comtet1974advanced}. Cauchy's estimate on any fixed circle of radius smaller than $\log2$ gives the claimed factorial bound.
\end{proof}

\begin{lemma}
\label{lem:common-base-repeated-pair}
Let $m\geq0$, let $G\sim N(0,I_r)$, and let $H$ be independent of $G$. For $i=1,\ldots,n$, let $Z_i,Z_{i,1},\ldots,Z_{i,m}\in\mathbb R^r$ be jointly distributed random vectors independent of $(H,G)$, without requiring independence among the displayed vectors. Suppose that, for some $K,c_1,\ldots,c_m\geq0$,
\begin{align*}
\sum_{i=1}^n\mathbb E[Z_iZ_i^\top]
\preceq
KI_r,
\quad
\max_{j\leq2d}|\langle u_j,Z_{i,\ell}\rangle|
\leq
c_\ell
\quad\text{a.s.}
\end{align*}
Then, for every $\sigma>0$,
\begin{align}
\left|
\sum_{i=1}^n
\mathbb E D^{m+2}\pi_{\beta,a_0,a}(H+\sigma G)
[Z_i,Z_i,Z_{i,1},\ldots,Z_{i,m}]
\right|\leq
\frac{CK}{\sigma^2}
C_1^m m!(m+1)\beta^m\log(4d)
\prod_{\ell=1}^m c_\ell,
\label{eq:common-base-repeated-pair}
\end{align}
where an empty product equals one and $C,C_1>0$ are universal constants.
\end{lemma}

\begin{proof}
Condition on $H$ and all direction arrays. Two Gaussian integrations by parts give the exact identity
\begin{align*}
&\mathbb E_G D^{m+2}\pi_{\beta,a_0,a}(H+\sigma G)
[Z_i,Z_i,Z_{i,1},\ldots,Z_{i,m}]\\
&\qquad=
\sigma^{-2}\mathbb E_G\left[
\operatorname{tr}\{H_2(G)Z_iZ_i^\top\}
D^m\pi_{\beta,a_0,a}(H+\sigma G)
[Z_{i,1},\ldots,Z_{i,m}]
\right].
\end{align*}
We then average over the direction arrays and apply the exact formula in Lemma~\ref{lem:selector-derivatives} to the last derivative. Fix a partition $\mathcal Q\in\mathfrak P_m$ and put $M=|\mathcal Q|$. Introduce $M+1$ independent Gumbel arrays: one for the factor $\pi_0$ and one for each block $B\in\mathcal Q$. Conditional on $H$ and these offsets, the joint vector of winner labels partitions the Gaussian variable $G$ into at most
\begin{align*}
N_{\mathcal Q}
&\leq
(2d+1)^{M+1}
\end{align*}
convex Borel cells, modulo Gaussian-null tie boundaries. Convexity follows because each winner event is an intersection of affine halfspaces and intersections of such winner events remain convex.

On a cell with winner labels $j_0$ and $(j_B:B\in\mathcal Q)$, set the matrix weight equal to zero if $j_0\neq0$ and otherwise set
\begin{align*}
A_{\mathcal Q,\boldsymbol j}
&=
\sum_{i=1}^n
\mathbb E\left[
Z_iZ_i^\top
\prod_{B\in\mathcal Q}
\prod_{\ell\in B}
\langle u_{j_B},Z_{i,\ell}\rangle
\right],
\end{align*}
with $u_0=0$. For every unit vector $v$,
\begin{align*}
|v^\top A_{\mathcal Q,\boldsymbol j}v|
&\leq
\left(\prod_{\ell=1}^m c_\ell\right)
\sum_{i=1}^n\mathbb E(v^\top Z_i)^2
\leq
K\prod_{\ell=1}^m c_\ell.
\end{align*}
Hence $\|A_{\mathcal Q,\boldsymbol j}\|_{\mathrm{op}}\leq K\prod_\ell c_\ell$. Lemma~\ref{lem:convex-cell-contraction}, applied conditionally to the joint-winner partition, bounds the contribution associated with $\mathcal Q$ by
\begin{align*}
4K\left(\prod_{\ell=1}^m c_\ell\right)
\log N_{\mathcal Q}
&\leq
CK(M+1)\log(4d)
\prod_{\ell=1}^m c_\ell.
\end{align*}
The Gumbel-max formula shows that averaging over the offsets reproduces exactly the product of softmax probabilities in Lemma~\ref{lem:selector-derivatives}. Multiplying by $\beta^mM!$, summing over $\mathcal Q$, and using
\begin{align*}
\sum_{\mathcal Q\in\mathfrak P_m}(|\mathcal Q|+1)|\mathcal Q|!
&\leq
(m+1)C_1^m m!
\end{align*}
from Lemma~\ref{lem:ordered-Bell-bound} proves Eq.~\eqref{eq:common-base-repeated-pair}. All conditioning and averaging steps are justified by the bounded selector derivatives and the displayed covariance and projection bounds.
\end{proof}

\section{Supporting results for Section~\ref{sec:S-to-T}}
\label{sec:S-to-T-appendix}

\begin{lemma}
\label{lem:path-moments}
For the variables $S$, $T$, and $S_s$ in Section~{\color{black}\ref{sec:S-to-T}},
\begin{align*}
\mathbb E[S]
=\mathbb E[T]=0,
\quad
\operatorname{Cov}(S)
=\operatorname{Cov}(T)=I_r,
\quad
\mathbb E[S^{\otimes3}]
=\mathbb E[T^{\otimes3}],
\end{align*}
and, for every $s\in[0,1]$,
\begin{align*}
\mathbb E[S_s]
=0,
\quad
\operatorname{Cov}(S_s)
=I_r,
\quad
\mathbb E[S_s^{\otimes3}]
=\mathbb E[S^{\otimes3}].
\end{align*}
\end{lemma}

\begin{proof}
Since $\mathbb E(2D_i)^2=1/2$ and $\mathbb E(2D_i)^3=1$,
\begin{align*}
\operatorname{Cov}(T)
&=
\frac12I_r+\frac12\sum_{i=1}^n\mathbb E[Y_iY_i^\top]
=
I_r,
\\
\mathbb E[T^{\otimes3}]
&=
\sum_{i=1}^n\mathbb E[(2D_iY_i')^{\otimes3}]
=
\sum_{i=1}^n\mathbb E[Y_i^{\otimes3}]
=
\mathbb E[S^{\otimes3}].
\end{align*}
For the path, $\mathbb E[B_i(s)]=(1-s/2)^3$, and hence
\begin{align*}
\mathbb E\left\{\frac{B_i(s)}{1-s/2}\right\}^2
=
1-s/2,
\quad
\mathbb E\left\{\frac{B_i(s)}{1-s/2}\right\}^3
=
1.
\end{align*}
Therefore,
\begin{align*}
\operatorname{Cov}(S_s)
=
\frac{s}{2}I_r+(1-s/2)\sum_{i=1}^n\mathbb E[Y_iY_i^\top]
=
I_r,
\quad
\mathbb E[S_s^{\otimes3}]
=
\sum_{i=1}^n\mathbb E[Y_i^{\otimes3}].
\end{align*}
Centering and independence eliminate all cross-row third moments, proving the result.
\end{proof}

\begin{lemma}
\label{lem:exact-Bernoulli-generator}
Suppose that $\mathbb E\|Y_i\|_2^4<\infty$, for $i=1,\ldots,n$. For every \(f\in C_b^4(\mathbb R^r)\) and \(s\in(0,1)\),
the derivative of \(\mathbb E[f(S_s)]\) is the right-hand side of
Eq.~\eqref{eq:main-Bernoulli-generator}, with
\(\pi_{\beta,a_0,a}\) replaced by \(f\). Moreover, the map $s\longmapsto \mathbb E[f(S_s)]$ is Lipschitz, and hence absolutely continuous, on \([0,1]\).
The right-hand side of the derivative identity is bounded on
\((0,1)\), and therefore belongs to \(L^1(0,1)\). Consequently,
\begin{align}
\mathbb E[f(S_1)]-\mathbb E[f(S_0)]
&=
\int_0^1
\frac{d}{ds}\mathbb E[f(S_s)]\,ds.
\label{eq:integrated-Bernoulli-generator}
\end{align}
\end{lemma}

\begin{proof}
Differentiating the Gaussian variance $s/2$, the Bernoulli probability $(1-s/2)^3$, and the scale $(1-s/2)^{-1}$ gives
\begin{align}
\frac{d}{ds}\mathbb E[f(S_s)]
={}&
\frac14\mathbb E[\Delta f(S_s)]
+
\frac{(1-s/2)^2}{2}\sum_{i=1}^n\mathbb E\left[Df\left(R_{i,s}+\frac{Y_i}{1-s/2}\right)\left[\frac{Y_i}{1-s/2}\right]\right]
\nonumber\\
&-
\frac{3(1-s/2)^2}{2}\sum_{i=1}^n\mathbb E\left[f\left(R_{i,s}+\frac{Y_i}{1-s/2}\right)-f(R_{i,s})\right].
\label{eq:raw-Bernoulli-derivative}
\end{align}
Let $Y_i'$ be an independent copy of $Y_i$. Since $\sum_i\mathbb E[Y_iY_i^\top]=I_r$, conditioning on $B_i(s)$ and factoring $(1-s/2)^2/2$ shows that the $i$th contribution in Eq.~\eqref{eq:raw-Bernoulli-derivative} equals
\begin{align}
{}&\frac{1-(1-s/2)^3}{2}\mathbb E D^2f(R_{i,s})\left[\frac{Y_i'}{1-s/2},\frac{Y_i'}{1-s/2}\right]
\nonumber\\
&+\frac{(1-s/2)^3}{2}\mathbb E D^2f\left(R_{i,s}+\frac{Y_i}{1-s/2}\right)\left[\frac{Y_i'}{1-s/2},\frac{Y_i'}{1-s/2}\right]
\nonumber\\
&+\mathbb E Df\left(R_{i,s}+\frac{Y_i}{1-s/2}\right)\left[\frac{Y_i}{1-s/2}\right]
-3\mathbb E\left[f\left(R_{i,s}+\frac{Y_i}{1-s/2}\right)-f(R_{i,s})\right].
\label{eq:row-Bernoulli-derivative}
\end{align}
Taylor's formula gives
\begin{align*}
&D^2f(x+v)[w,w]
={}
D^2f(x)[w,w]+D^3f(x)[v,w,w]
+
\int_0^1(1-t)D^4f(x+tv)[v,v,w,w]dt,
\\
&Df(x+v)[v]
={}
Df(x)[v]+D^2f(x)[v,v]+\frac12D^3f(x)[v,v,v]
+
\frac12\int_0^1(1-t)^2D^4f(x+tv)[v,v,v,v]dt,
\\
&f(x+v)-f(x)
={}
Df(x)[v]+\frac12D^2f(x)[v,v]+\frac16D^3f(x)[v,v,v]
+
\frac16\int_0^1(1-t)^3D^4f(x+tv)[v,v,v,v]dt.
\end{align*}
The linear terms vanish after expectation because $Y_i$ is centered and independent of $R_{i,s}$. The mixed cubic term vanishes because $Y_i$ and $Y_i'$ are independent and centered. The quadratic coefficient satisfies
\begin{align*}
\frac{1-(1-s/2)^3}{2}+\frac{(1-s/2)^3}{2}+1-\frac32
&=0,
\end{align*}
and the pure cubic coefficient satisfies
\begin{align*}
\frac12-\frac{3}{6}
&=0.
\end{align*}
Thus only the fourth-order integral remainders remain. Since
\begin{align*}
(1-t)^2-(1-t)^3
&=t(1-t)^2,
\end{align*}
collecting their coefficients gives Eq.~\eqref{eq:main-Bernoulli-generator}.

It remains to justify the endpoint integration. Since
\(1-s/2\geq1/2\) for every \(s\in[0,1]\), the final fourth-order
generator identity and the boundedness of \(D^4f\) imply
\begin{align}
\sup_{0<s<1}
\left|
\frac{d}{ds}\mathbb E[f(S_s)]
\right|
&\leq
C
\sup_{x\in\mathbb R^r}
\|D^4f(x)\|_{\mathrm{op}}
\sum_{i=1}^n
\left\{
\mathbb E\|Y_i\|_2^4
+
\bigl(\mathbb E\|Y_i\|_2^2\bigr)^2
\right\}
\nonumber\\
&<\infty.
\label{eq:Bernoulli-generator-uniform-bound}
\end{align}
Indeed, the pure fourth-order terms are bounded by a constant
multiple of \(\mathbb E\|Y_i\|_2^4\), while the mixed terms involving
an independent copy \(Y_i'\) are bounded by
\[
\mathbb E\bigl[\|Y_i\|_2^2\|Y_i'\|_2^2\bigr]
=
\bigl(\mathbb E\|Y_i\|_2^2\bigr)^2.
\]

Hence, by the mean value theorem, for every \(0<s<t<1\),
\[
\left|
\mathbb E[f(S_t)]-\mathbb E[f(S_s)]
\right|
\leq
M_f|t-s|,
\]
where \(M_f<\infty\) denotes the right-hand side of
Eq.~\eqref{eq:Bernoulli-generator-uniform-bound}.

To handle the endpoints, realize the entire Bernoulli path on one
probability space. Let \(U_1,\ldots,U_n\) be independent
\(\operatorname{Uniform}(0,1)\) random variables, jointly independent
of \(G\), \((Y_i)_{i=1}^n\), and \((Y_i')_{i=1}^n\), and put
\[
B_i(s)
=
\mathbf 1\!\left\{
U_i\leq(1-s/2)^3
\right\},
\qquad
0\leq s\leq1.
\] For each fixed \(s_0\in[0,1]\),
\[
B_i(s)\longrightarrow B_i(s_0)
\quad\text{almost surely as }s\to s_0,
\]
except on the null event
\(\{U_i=(1-s_0/2)^3\}\). Since all remaining coefficients in
\(S_s\) are continuous in \(s\), it follows that
\[
S_s\longrightarrow S_{s_0}
\quad\text{almost surely as }s\to s_0.
\]
Because \(f\) is bounded and continuous, dominated convergence gives
\[
\mathbb E[f(S_s)]
\longrightarrow
\mathbb E[f(S_{s_0})].
\]
Thus \(s\mapsto\mathbb E[f(S_s)]\) is continuous at \(0\) and \(1\).

Letting \(s\downarrow0\) and \(t\uparrow1\) in the preceding
Lipschitz inequality shows that
\(s\mapsto\mathbb E[f(S_s)]\) is Lipschitz on all of \([0,1]\).
It is therefore absolutely continuous, and the fundamental theorem
of calculus yields
Eq.~\eqref{eq:integrated-Bernoulli-generator}.
\end{proof}

\begin{lemma}
\label{lem:finite-ghost-completion}
Let \(K\geq1\) be an integer. Let \(R\) be a random vector, let
\(\eta\) be a centered random vector, and let $\eta^{(0)},\eta^{(1)},\ldots,\eta^{(K)}$ be jointly independent copies of \(\eta\), with the entire copy family
jointly independent of \(R\). Define $Pf(x)
:=
\mathbb E[f(x+\eta)]$ and $H:=R+\eta^{(0)}$.

\begin{enumerate}[label=(\roman*)]

\item
For every bounded Borel-measurable function \(f\),
\begin{align}
\mathbb E[f(R)]
&=
\sum_{k=0}^{K-1}
\mathbb E\left[(I-P)^k f(H)\right]
+
\mathbb E\left[(I-P)^K f(R)\right],
\label{eq:finite-ghost-completion}
\end{align}
where \((I-P)^0=I\).

\item
Fix \(k\in\{1,\ldots,K\}\). Suppose that
\(f\in C^{2k}(\mathbb R^r)\), that all derivatives of \(f\) through
order \(2k\) are bounded, and that $\mathbb E\|\eta\|_2^2<\infty$. Then, for every \(x\in\mathbb R^r\),
\begin{align}
(I-P)^k f(x)
={}&
(-1)^k
\mathbb E
\int_{[0,1]^k}
\prod_{\ell=1}^k(1-v_\ell)
\nonumber\\
&\quad\times
D^{2k}f\left(
x+\sum_{\ell=1}^k v_\ell\eta^{(\ell)}
\right)
\left[
(\eta^{(1)})^{\otimes2},
\ldots,
(\eta^{(k)})^{\otimes2}
\right]
\,dv_1\cdots dv_k.
\label{eq:differentiated-ghost-completion}
\end{align}
The random integral in
\eqref{eq:differentiated-ghost-completion} is absolutely integrable.

\end{enumerate}
\end{lemma}

\begin{proof}
Let \(\mu_-\) be the law of \(R\), and let
\(\mu=\mu_-P\) be the law of \(H\). Since $P=I-(I-P)$, for every \(k\geq0\),
\begin{align*}
\mu(I-P)^k f
&=
\mu_-P(I-P)^k f\\
&=
\mu_-(I-P)^k f
-
\mu_-(I-P)^{k+1}f.
\end{align*}
Summing this identity over \(k=0,\ldots,K-1\) gives
\eqref{eq:finite-ghost-completion}.

For part (ii), centering of \(\eta\) and the second-order Taylor
formula give
\begin{align*}
(I-P)f(x)
&=
f(x)-\mathbb E[f(x+\eta)]\\
&=
-\mathbb E\int_0^1
(1-v)
D^2f(x+v\eta)[\eta,\eta]
\,dv.
\end{align*}
This proves the formula for \(k=1\).

Suppose that the formula holds for some \(k<K\). Apply \(I-P\)
once more, using the independent copy \(\eta^{(k+1)}\), to the
function $x\longmapsto
D^{2k}f(x)
\left[
(\eta^{(1)})^{\otimes2},
\ldots,
(\eta^{(k)})^{\otimes2}
\right]$. The preceding second-order identity then gives the formula at order
\(k+1\). Induction proves
\eqref{eq:differentiated-ghost-completion}.

Finally,
\begin{align*}
&\mathbb E
\int_{[0,1]^k}
\prod_{\ell=1}^k(1-v_\ell)
\left|
D^{2k}f\left(
x+\sum_{\ell=1}^k v_\ell\eta^{(\ell)}
\right)
\left[
(\eta^{(1)})^{\otimes2},
\ldots,
(\eta^{(k)})^{\otimes2}
\right]
\right|
\,dv
\\
&\leq
2^{-k}
\sup_{y\in\mathbb R^r}
\|D^{2k}f(y)\|_{\mathrm{op}}
\prod_{\ell=1}^k
\mathbb E\|\eta^{(\ell)}\|_2^2
\\
&=
2^{-k}
\sup_{y\in\mathbb R^r}
\|D^{2k}f(y)\|_{\mathrm{op}}
\left(\mathbb E\|\eta\|_2^2\right)^k
<\infty.
\end{align*}
Thus all applications of Fubini and differentiation are justified.
\end{proof}

\begin{definition}
\label{def:common-terminal-contributions}

\begin{enumerate}
    \item Fix the data of Proposition~\ref{prop:S-to-T}, an even integer
\(Q\geq8\), and \(s\in(0,1]\). Put $a_s:=1-s/2$ and $K_Q:=(Q-4)/2$, and define
\[
c_1(s):=\frac{a_s^2}{4},
\qquad
c_2(s):=\frac{a_s^5}{4},
\]
together with $w_1(t):=t(1-t)^2$ and $w_2(t):=1-t$.

\item For each \(i\), put $\eta_i:=B_i(s)Y_i/a_s$ and $\mu_i:=\mathcal L(\eta_i)$. On an enlarged probability space, let $\eta_i^{(0)},\eta_i^{(1)},\ldots,\eta_i^{(K_Q)}$ have common law \(\mu_i\).  The full ghost family is taken to be
mutually independent and independent of the original variables $G$ and $\{B_j(s),Y_j,Y_j':1\leq j\leq n\}$.

\item The operator \(P_i\) convolves only the spatial variable.  More
precisely, for a measurable function \(g=g(x;y,y')\), define
\[
(P_i g)(x;y,y')
:=
\int_{\mathbb R^r} g(x+z;y,y')\,\mu_i(dz),
\]
whenever the integral exists.  Thus \(y,y'\) are held fixed when
\(P_i\) is applied; equivalently, \(P_i\) averages only over a fresh
ghost increment having law \(\mu_i\).

\item For \(t\in[0,1]\), define
\begin{align*}
f_{i,t}^{(1)}(x;y,y')
&:=
D^4\pi_{\beta,a_0,a}
\left(x+t\frac{y}{a_s}\right)
\left[
\left(\frac{y}{a_s}\right)^{\otimes4}
\right],
\\
f_{i,t}^{(2)}(x;y,y')
&:=
D^4\pi_{\beta,a_0,a}
\left(x+t\frac{y}{a_s}\right)
\left[
\frac{y}{a_s},
\frac{y}{a_s},
\frac{y'}{a_s},
\frac{y'}{a_s}
\right].
\end{align*}
In subsequent formulas these functions are evaluated at
\((y,y')=(Y_i,Y_i')\); these two arguments may be suppressed when
there is no ambiguity.

\item For \(k\in\{0,\ldots,K_Q-1\}\) and
\(v=(v_1,\ldots,v_k)\in[0,1]^k\), set
\[
L_{i,k}(t,v)
:=
t\frac{Y_i}{a_s}
+
\sum_{\ell=1}^k
v_\ell\eta_i^{(\ell)}.
\]

\item Define the two direction lists
\begin{align*}
\mathcal D_{i,k}^{(1)}
&:=
\left[
\left(\frac{Y_i}{a_s}\right)^{\otimes4},
(\eta_i^{(1)})^{\otimes2},
\ldots,
(\eta_i^{(k)})^{\otimes2}
\right],
\\
\mathcal D_{i,k}^{(2)}
&:=
\left[
\frac{Y_i}{a_s},
\frac{Y_i}{a_s},
\frac{Y_i'}{a_s},
\frac{Y_i'}{a_s},
(\eta_i^{(1)})^{\otimes2},
\ldots,
(\eta_i^{(k)})^{\otimes2}
\right].
\end{align*}

\item Let \(H_s\) be a random vector having the common law of $R_{i,s}+\eta_i^{(0)}$ and jointly independent of all displayed row directions and ghost
copies. This law does not depend on \(i\).

\item For \(q\in\{4,\ldots,Q-1\}\), define
\begin{align}
\mathcal N_{Q,q}(s)
:={}&
\sum_{\ell=1}^2
c_\ell(s)
\sum_{\substack{
0\leq k<K_Q,\ h\geq0\\
4+2k+h=q
}}
\frac{(-1)^k}{h!}
\sum_{i=1}^n
\mathbb E
\int_0^1
\int_{[0,1]^k}
w_\ell(t)
\prod_{j=1}^k(1-v_j)
\nonumber\\
&\qquad\qquad\times
D^q\pi_{\beta,a_0,a}(H_s)
\left[
\mathcal D_{i,k}^{(\ell)},
L_{i,k}(t,v)^{\otimes h}
\right]
\,dv\,dt.
\label{eq:definition-NQq}
\end{align}

\item Here and below, when \(k=0\), the integral over \([0,1]^k\) and the
product over \(j\) are interpreted as \(1\). Define
\begin{align}
\mathcal N_Q(s)
:=
\sum_{q=4}^{Q-1}\mathcal N_{Q,q}(s).
\label{eq:definition-NQ}
\end{align}

\item For \(0\leq k<K_Q\), put $p_k:=Q-4-2k$. Define the corresponding order-\(Q\) Taylor remainder by
\begin{align*}
\mathcal R_{i,k}^{(\ell)}(s,t,v)
:={}&
\frac{(-1)^k}{(p_k-1)!}
\int_0^1
(1-u)^{p_k-1}
\\
&\quad\times
D^Q\pi_{\beta,a_0,a}
\left(
H_s+uL_{i,k}(t,v)
\right)
\left[
\mathcal D_{i,k}^{(\ell)},
L_{i,k}(t,v)^{\otimes p_k}
\right]
\,du.
\end{align*}

\item Define
\begin{align}
\mathcal T_Q(s)
:={}&
\sum_{\ell=1}^2
c_\ell(s)
\sum_{i=1}^n
\sum_{k=0}^{K_Q-1}
\mathbb E
\int_0^1
\int_{[0,1]^k}
w_\ell(t)
\prod_{j=1}^k(1-v_j)
\mathcal R_{i,k}^{(\ell)}(s,t,v)
\,dv\,dt
\nonumber\\
&+
\sum_{\ell=1}^2
c_\ell(s)
\sum_{i=1}^n
\mathbb E
\int_0^1
w_\ell(t)
\left[(I-P_i)^{K_Q}f_{i,t}^{(\ell)}\right]
   (R_{i,s};Y_i,Y_i')
\,dt.
\label{eq:definition-TQ}
\end{align}
Thus the first line of \(\mathcal T_Q(s)\) collects the exact
order-\(Q\) Taylor remainders, while the second line collects the
terminal ghost remainders.
\end{enumerate}

\end{definition}

\begin{lemma}
\label{lem:nonterminal-common-base}
There exist universal constants \(C,C^*>0\) such that the following
holds. Under the assumptions of Proposition~\ref{prop:S-to-T}, let
\(Q\geq8\) be even, let \(s\in[\tau,1]\), and suppose $C^*Q\beta b_{d,n}\leq 1/4$. Then, for every \(q=4,\ldots,Q-1\),
\begin{align}
|\mathcal N_{Q,q}(s)|
&\leq
\frac{
Cb_{d,n}^2\beta^2\log(4d)
}{s}
q^5
(C^*q\beta b_{d,n})^{q-4}.
\label{eq:nonterminal-order-q}
\end{align}
Consequently,
\begin{align}
|\mathcal N_Q(s)|
&\leq
\frac{
Cb_{d,n}^2\beta^2\log(4d)
}{s}.
\label{eq:nonterminal-total}
\end{align}
\end{lemma}

\begin{proof}
At ghost depth $k$, the first direction pattern gives an average of
\begin{align*}
D^{4+2k}\pi_{\beta,a_0,a}\left(H+t\frac{Y_i}{1-s/2}+\sum_{\ell=1}^kv_\ell\eta_i^{(\ell)}\right)
\left[\left(\frac{Y_i}{1-s/2}\right)^{\otimes4},(\eta_i^{(1)})^{\otimes2},\ldots,(\eta_i^{(k)})^{\otimes2}\right],
\end{align*}
where $H$ has the common full-sum law and is independent of all displayed directions. The second pattern replaces the first four directions by
\begin{align*}
\left[\frac{Y_i}{1-s/2},\frac{Y_i}{1-s/2},\frac{Y_i'}{1-s/2},\frac{Y_i'}{1-s/2}\right].
\end{align*}
Taylor expansion around $H$ has explicit degrees $h=0,\ldots,Q-5-2k$, and every explicit term has total order $q=4+2k+h<Q$. Expanding the $h$ local-shift directions produces at most $(k+1)^h$ monomials.

If $k=0$, select two copies of $Y_i/(1-s/2)$ in the first pattern and two copies of $Y_i'/(1-s/2)$ in the second. If $k\geq1$, select one ghost pair. Their aggregate covariance matrices satisfy
\begin{align*}
\sum_{i=1}^n\mathbb E\left[\frac{Y_i}{1-s/2}\left(\frac{Y_i}{1-s/2}\right)^\top\right]
\preceq4I_r,
\quad
\sum_{i=1}^n\mathbb E[\eta_i^{(\ell)}\eta_i^{(\ell)\top}]
=(1-s/2)I_r\preceq I_r.
\end{align*}
Every remaining direction has defining projection at most $2{\color{black}b_{d,n}}$, and the local shift satisfies
\begin{align*}
\max_{j\leq2d}\left|\left\langle u_j,t\frac{Y_i}{1-s/2}+\sum_{\ell=1}^kv_\ell\eta_i^{(\ell)}\right\rangle\right|
&\leq
2(k+1){\color{black}b_{d,n}}.
\end{align*}
Conditional on the non-Gaussian part of $H$, the Gaussian component is $\sqrt{s/2}G$. Two Gaussian integrations by parts in the selected repeated pair give Eq.~\eqref{eq:two-IBP-main}.

Fix one monomial and denote the selected repeated pair by $Z_i,Z_i$ and the remaining directions by $Z_{i,3},\ldots,Z_{i,q}$. After applying Lemma~\ref{lem:selector-derivatives}, every derivative term is a sum of products of softmax probabilities and directional scores. Lemma~\ref{lem:Gumbel-cells} represents each product of softmax probabilities by joint winner events forming a convex partition. After summing over $i$ and before taking operator norms, the matrix weights satisfy
\begin{align*}
\left\|
\sum_{i=1}^n\mathbb E\left[Z_iZ_i^\top\prod_{\ell=3}^q\langle u_{j_\ell},Z_{i,\ell}\rangle\right]
\right\|_{\mathrm{op}}
&\leq
C^q {\color{black}b_{d,n}}^{q-2}.
\end{align*}
Lemma~\ref{lem:convex-cell-contraction} therefore gives a factor $C^q\log(4d)/s$ after the two integrations by parts.

The selector derivative coefficient, the $k$ ghost pairs, the $(k+1)^h$ shift expansion, and the factor $1/h!$ from Taylor's formula give, relative to $C{\color{black}b_{d,n}}^2\beta^2\log(4d)/s$,
\begin{align*}
\frac{C_0^q q!}{h!}(\beta {\color{black}b_{d,n}})^{2k}\{(k+1)\beta {\color{black}b_{d,n}}\}^h
&\leq
Cq^4(C^*q\beta {\color{black}b_{d,n}})^{q-4},
\end{align*}
where Lemma~\ref{lem:ordered-Bell-bound} is used in the first factor and $q=4+2k+h$ in the second inequality. For each $q$, there are at most $q$ pairs $(k,h)$, so
\begin{align*}
\sum_{q=4}^{Q-1}q^5(C^*q\beta {\color{black}b_{d,n}})^{q-4}
&\leq
\sum_{q=4}^\infty q^5 4^{-(q-4)}
\leq
C.
\end{align*}

For completeness, we now verify the full common-base term and the coefficient count. Write the common state as
\begin{align*}
H_s
&=
H_{0,s}+\sqrt{\frac{s}{2}}G,
\end{align*}
where $H_{0,s}$ is independent of $G$ and of every displayed row direction. At ghost depth $k$, let
\begin{align*}
L_{i,k}
&=
t\frac{Y_i}{1-s/2}+\sum_{\ell=1}^kv_\ell\eta_i^{(\ell)}.
\end{align*}
Taylor expansion about $H_s$ shows that a term of degree $h$ has the form
\begin{align}
\frac{1}{h!}
\sum_{i=1}^n
\mathbb E D^q\pi_{\beta,a_0,a}(H_s)
[\mathcal D_{i,k},L_{i,k}^{\otimes h}],
\qquad
q=4+2k+h<Q,
\label{eq:explicit-common-base-term}
\end{align}
up to a sign and integration weights whose total absolute mass is at most one. Here $\mathcal D_{i,k}$ consists of one of the two original four-direction patterns followed by $k$ ghost pairs.

When $k=0$, select a repeated pair from the original pattern: two copies of $Y_i/(1-s/2)$ in the first pattern and the two copies of $Y_i'/(1-s/2)$ in the second. When $k\geq1$, select the first ghost pair. Denote the selected direction by $Z_i$. In every case,
\begin{align*}
\sum_{i=1}^n\mathbb E[Z_iZ_i^\top]
&\preceq
4I_r.
\end{align*}
Every remaining original or ghost direction has defining projection at most $2b_{d,n}$, while
\begin{align*}
\max_{j\leq2d}|\langle u_j,L_{i,k}\rangle|
&\leq
2(k+1)b_{d,n}.
\end{align*}
Applying Lemma~\ref{lem:common-base-repeated-pair} to Eq.~\eqref{eq:explicit-common-base-term}, with $m=q-2$ and $\sigma^2=s/2$, bounds a fixed $(k,h)$ contribution by
\begin{align*}
\frac{Cb_{d,n}^2\beta^2\log(4d)}{s}
\frac{C_2^q(q-2)!(q-1)}{h!}
(k+1)^h
(\beta b_{d,n})^{q-4}.
\end{align*}
Since $q-2-h=2k+2$ and $k+1\leq q$,
\begin{align*}
\frac{(q-2)!(q-1)}{h!}(k+1)^h
&\leq
q^{2k+2}q q^h
=
q^{q-1}
=
q^3q^{q-4}.
\end{align*}
After enlarging $C^*$, this is bounded by
\begin{align*}
\frac{Cb_{d,n}^2\beta^2\log(4d)}{s}
q^3(C^*q\beta b_{d,n})^{q-4}.
\end{align*}
For a fixed $q$, there are at most $q$ admissible pairs $(k,h)$, two generator patterns, and only a fixed number of choices of the repeated pair. Thus Eq.~\eqref{eq:nonterminal-order-q} holds. Summing over
\(q=4,\ldots,Q-1\), and using
\[
\sum_{q=4}^{Q-1}
q^5(C^*q\beta b_{d,n})^{q-4}
\leq C,
\]
proves Eq.~\eqref{eq:nonterminal-total}.
\end{proof}

\begin{lemma}
\label{lem:terminal-order-Q}
There exist universal constants \(C,C^*>0\) such that the following
holds. Under the assumptions of Proposition~\ref{prop:S-to-T}, let
\(Q\geq8\) be even, let \(s\in[\tau,1]\), and suppose $C^*Q\beta b_{d,n}\leq 1/4$. Then the quantity \(\mathcal T_Q(s)\) defined in
Eq.~\eqref{eq:definition-TQ} satisfies
\begin{align}
|\mathcal T_Q(s)|
&\leq
CnQ^2(C^*Q\beta b_{d,n})^Q.
\label{eq:terminal-total}
\end{align}
\end{lemma}

\begin{proof}
At ghost depth $k$, Taylor's formula leaves an exact derivative of order $Q$ with four original directions, $k$ ghost pairs, and $Q-4-2k$ copies of the local shift. Lemma~\ref{lem:selector-derivatives} gives, for each row and depth,
\begin{align*}
\frac{C_0^Q Q!}{(Q-4-2k)!}\beta^Q(2{\color{black}b_{d,n}})^{4+2k}\{2(k+1){\color{black}b_{d,n}}\}^{Q-4-2k}
&\leq
(C^*Q\beta {\color{black}b_{d,n}})^Q.
\end{align*}
The same estimate applies to the terminal ghost term because its derivative order is exactly $Q$. 

Here is the explicit enumeration. At depth $0\leq k<(Q-4)/2$, let $m_k=4+2k$ and $p_k=Q-m_k$. For either original direction pattern, Taylor's integral remainder is
\begin{align*}
\frac{1}{(p_k-1)!}
\int_0^1(1-v)^{p_k-1}
D^Q\pi_{\beta,a_0,a}(H_s+vL_{i,k})
[\mathcal D_{i,k},L_{i,k}^{\otimes p_k}]dv.
\end{align*}
The integral weight has mass $1/p_k!$. Lemma~\ref{lem:selector-derivatives}, the bounds $\max_j|\langle u_j,Z\rangle|\leq2b_{d,n}$ for every direction in $\mathcal D_{i,k}$, and $\max_j|\langle u_j,L_{i,k}\rangle|\leq2(k+1)b_{d,n}$ give the upper bound of the absolute value of Taylor remainder at depth $k < (Q-4)/2$ and row $i$
\begin{align*}
\frac{C_0^QQ!}{p_k!}\beta^Q
(2b_{d,n})^{m_k}
\{2(k+1)b_{d,n}\}^{p_k}\leq
(C^*Q\beta b_{d,n})^Q,
\end{align*}
because $Q!/p_k!\leq Q^{m_k}$ and $(k+1)^{p_k}\leq Q^{p_k}$. At the terminal ghost depth $(Q-4)/2$, Lemma~\ref{lem:finite-ghost-completion} gives a derivative of order $Q$ with four original directions and $(Q-4)/2$ ghost pairs. Its absolute value is at most
\begin{align*}
C_0^QQ!\beta^Q(2b_{d,n})^Q
&\leq
(C^*Q\beta b_{d,n})^Q.
\end{align*}
All $t$, $v$, and ghost-integration weights have total absolute mass at most one. Summing over the two generator patterns, fewer than $Q$ depths, and $n$ active rows costs at most $CnQ$; the stated $CnQ^2$ bound is therefore valid and leaves room for the finite direction allocations. There are $n$ choices of the active row, fewer than $Q$ ghost depths, and the finite allocation and integration sums cost at most one additional factor $Q$. Summing over the two generator patterns, the \(n\) active rows, and
the fewer than \(Q\) admissible ghost depths proves
\[
|\mathcal T_Q(s)|
\leq
CnQ^2(C^*Q\beta b_{d,n})^Q,
\]
which is Eq.~\eqref{eq:terminal-total}.
\end{proof}

\section{Supporting results for Section~\ref{subsec:T-to-G}}
\label{sec:T-to-G-appendix}

\begin{lemma}
\label{lem:Gaussian-replacement-generator}
Assume that the collection $\left\{
Y_i',\widetilde Y_i,D_i',\widetilde D_i
:
1\leq i\leq n
\right\}$ is jointly independent and independent of $\sigma\!\left(
G,G_1,(D_j,Y_j)_{1\leq j\leq n}
\right)$, with $\mathcal L(Y_i')
=
\mathcal L(\widetilde Y_i)
=
\mathcal L(Y_i)$ and $\mathcal L(D_i')
=
\mathcal L(\widetilde D_i)
=
\mathcal L(D_i)$ for every \(i\). Suppose that $\mathbb E\|Y_i\|_2^4<\infty$, for $i=1,\ldots,n$. For every \(f\in C_b^4(\mathbb R^r)\) and \(s\in(0,1)\),
the derivative of \(\mathbb E[f(T_s)]\) is the right-hand side of
Eq.~\eqref{eq:main-Gaussian-generator}, with
\(\pi_{\beta,a_0,a}\) replaced by \(f\). Moreover, the map $s\longmapsto \mathbb E[f(T_s)]$ is Lipschitz, and hence absolutely continuous, on \([0,1]\).
The right-hand side of the derivative identity is bounded on
\((0,1)\), and therefore belongs to \(L^1(0,1)\). Consequently,
\begin{align}
\mathbb E[f(T_1)]-\mathbb E[f(T_0)]
&=
\int_0^1
\frac{d}{ds}\mathbb E[f(T_s)]\,ds.
\label{eq:integrated-Gaussian-generator}
\end{align}
\end{lemma}

\begin{proof}
Differentiating $T_s$ and applying Gaussian integration by parts to $G_1$ gives
\begin{align}
\frac{d}{ds}\mathbb E[f(T_s)]
&=
\frac14\mathbb E[\Delta f(T_s)]
-
\frac{1}{2\sqrt{1-s}}
\sum_{i=1}^n\mathbb E\left[Df\left(R_{i,s}+2\sqrt{1-s}D_iY_i\right)[2D_iY_i]\right].
\label{eq:raw-Gaussian-replacement}
\end{align}
Taylor's formula gives
\begin{align}
Df\left(R_{i,s}+2\sqrt{1-s}D_iY_i\right)[2D_iY_i]
={}&
Df(R_{i,s})[2D_iY_i]
+
\sqrt{1-s}D^2f(R_{i,s})[2D_iY_i,2D_iY_i]
\nonumber\\
&+
\frac{1-s}{2}D^3f(R_{i,s})[(2D_iY_i)^{\otimes3}]
\nonumber\\
&+
\frac{(1-s)^{3/2}}{2}\int_0^1(1-t)^2D^4f\left(R_{i,s}+2t\sqrt{1-s}D_iY_i\right)
\nonumber\\
&\hspace{10em}\times[(2D_iY_i)^{\otimes4}]dt.
\label{eq:directional-Taylor-Gaussian}
\end{align}
The first term has expectation zero because $2D_iY_i$ is centered and independent of $R_{i,s}$. Since
\begin{align*}
\sum_{i=1}^n\mathbb E[(2D_iY_i)(2D_iY_i)^\top]
&=
\frac12I_r,
\end{align*}
we have
\begin{align*}
\frac14\mathbb E[\Delta f(T_s)]
&=
\frac12\sum_{i=1}^n\mathbb E\left[D^2f(T_s)[2D_i'Y_i',2D_i'Y_i']\right].
\end{align*}
Therefore, the difference between the Gaussian Laplacian term and the second-order Taylor term in Eq.~\eqref{eq:raw-Gaussian-replacement}, Eq.~\eqref{eq:directional-Taylor-Gaussian} is
\begin{align*}
{}&\frac12\sum_{i=1}^n\mathbb E\left[D^2f(T_s)[2D_i'Y_i',2D_i'Y_i']-D^2f(R_{i,s})[2D_i'Y_i',2D_i'Y_i']\right]
\nonumber\\
&=
\frac{1-s}{2}\sum_{i=1}^n\mathbb E\int_0^1(1-t)D^4f\left(R_{i,s}+2t\sqrt{1-s}D_iY_i\right)
[2D_iY_i,2D_iY_i,2D_i'Y_i',2D_i'Y_i']dt.
\end{align*}
The intervening third-order term contains one centered factor $2D_iY_i$, independent of $(R_{i,s},D_i',Y_i')$, and hence its expectation is zero.

For the third-order Taylor term, independence and $\mathbb E(2D_i)^3=1$ give
\begin{align*}
\mathbb E\left[D^3f(R_{i,s})[(2D_iY_i)^{\otimes3}]\right]
&=
\sum_{a_1,a_2,a_3=1}^r\mathbb E[Y_{ia_1}Y_{ia_2}Y_{ia_3}]\mathbb E[\partial_{a_1a_2a_3}f(R_{i,s})].
\end{align*}
Since $T_s=R_{i,s}+2\sqrt{1-s}D_iY_i$, Taylor's formula and an independent copy $(\widetilde D_i,\widetilde Y_i)$ give
\begin{align*}
{}&\sum_{a_1,a_2,a_3=1}^r\mathbb E[Y_{ia_1}Y_{ia_2}Y_{ia_3}]\mathbb E[\partial_{a_1a_2a_3}f(R_{i,s})]
\nonumber\\
&=
\sum_{a_1,a_2,a_3=1}^r\mathbb E[Y_{ia_1}Y_{ia_2}Y_{ia_3}]\mathbb E[\partial_{a_1a_2a_3}f(T_s)]
\nonumber\\
&\quad-
\sqrt{1-s}\mathbb E\int_0^1D^4f\left(R_{i,s}+2t\sqrt{1-s}D_iY_i\right)
[2D_iY_i,2\widetilde D_i\widetilde Y_i,2\widetilde D_i\widetilde Y_i,2\widetilde D_i\widetilde Y_i]dt.
\end{align*}
Substituting the preceding covariance and third moment identities into Eq.~\eqref{eq:raw-Gaussian-replacement}, and using the final fourth-order remainder in Eq.~\eqref{eq:directional-Taylor-Gaussian}, gives Eq.~\eqref{eq:main-Gaussian-generator}.

It remains to justify integration up to the endpoints. Although the
raw derivative in Eq.~\eqref{eq:raw-Gaussian-replacement} contains
the factor \((1-s)^{-1/2}\), that factor disappears after the
second- and third-order cancellations leading to
Eq.~\eqref{eq:main-Gaussian-generator}. By multilinearity of
\(D^3f\) and \(D^4f\), the final generator satisfies
\begin{align}
\sup_{0<s<1}
\left|
\frac{d}{ds}\mathbb E[f(T_s)]
\right|
&\leq
\frac14
\sup_{x\in\mathbb R^r}
\|D^3f(x)\|_{\mathrm{op}}
\sum_{i=1}^n
\mathbb E\|Y_i\|_2^3
\nonumber\\
&\quad+
C
\sup_{x\in\mathbb R^r}
\|D^4f(x)\|_{\mathrm{op}}
\sum_{i=1}^n
\mathbb E\|Y_i\|_2^4
<\infty.
\label{eq:Gaussian-generator-uniform-bound}
\end{align}
Indeed, the repeated-copy fourth-order terms are controlled using
\[
\left(\mathbb E\|Y_i\|_2^2\right)^2
\leq
\mathbb E\|Y_i\|_2^4,
\]
while the term containing one copy of \(Y_i\) and three copies of
\(\widetilde Y_i\) is controlled by
\[
\mathbb E\|Y_i\|_2\,
\mathbb E\|Y_i\|_2^3
\leq
\mathbb E\|Y_i\|_2^4.
\]

Next, realize all variables appearing in
\[
T_s
=
\frac{1}{\sqrt2}G
+
2\sqrt{1-s}\sum_{i=1}^nD_iY_i
+
\sqrt{\frac{s}{2}}G_1
\]
on the same probability space. For every outcome,
\(s\mapsto T_s\) is continuous on \([0,1]\). Since \(f\) is bounded
and continuous, dominated convergence gives
\[
\mathbb E[f(T_s)]
\longrightarrow
\mathbb E[f(T_{s_0})]
\qquad
\text{as }s\to s_0
\]
for every \(s_0\in[0,1]\).

For \(0<s<t<1\), the mean value theorem and
Eq.~\eqref{eq:Gaussian-generator-uniform-bound} give
\[
\left|
\mathbb E[f(T_t)]-\mathbb E[f(T_s)]
\right|
\leq
C_f|t-s|
\]
for some finite constant \(C_f\). Letting \(s\downarrow0\) or
\(t\uparrow1\), using the endpoint continuity just proved, shows
that the same bound holds for all \(s,t\in[0,1]\). Hence
\(s\mapsto\mathbb E[f(T_s)]\) is Lipschitz on \([0,1]\), and
therefore absolutely continuous. The fundamental theorem of calculus
now gives Eq.~\eqref{eq:integrated-Gaussian-generator}.
\end{proof}

\begin{lemma}
\label{lem:T-to-G-fourth-order}
There exist universal constants \(C,C^*>0\) such that the following
holds. Assume the probabilistic and selector hypotheses of
Proposition~\ref{prop:T-to-G}. Let \(Q\geq8\) be even and let
\(\beta>0\) satisfy $C^*Q\beta b_{d,n}\leq 1/4$. Then, uniformly over \(s\in[0,1]\), the absolute value of the sum of
the three fourth-order terms in
Eq.~\eqref{eq:main-Gaussian-generator} is at most
\[
C b_{d,n}^2\beta^2\log(4d)
+
C nQ^2(C^*Q\beta b_{d,n})^Q.
\]
\end{lemma}

\begin{proof}
The Gaussian part of $R_{i,s}$ satisfies
\begin{align*}
\frac{1}{\sqrt2}G+\sqrt{\frac{s}{2}}G_1
\stackrel d=
\sqrt{\frac{1+s}{2}}G,
\quad
\frac{1+s}{2}I_r
\succeq
\frac12I_r.
\end{align*}
The repeated-direction covariance budgets are
\begin{align*}
\sum_{i=1}^n\mathbb E[(2D_iY_i)(2D_iY_i)^\top]
=
\frac12I_r,
\quad
\sum_{i=1}^n\mathbb E[(2D_i'Y_i')(2D_i'Y_i')^\top]
=
\frac12I_r,
\end{align*}
and the same identity holds for $2\widetilde D_i\widetilde Y_i$. Moreover,
\begin{align*}
\max_{j\leq2d}|\langle u_j,2D_iY_i\rangle|,
\quad
\max_{j\leq2d}|\langle u_j,2D_i'Y_i'\rangle|,
\quad
\max_{j\leq2d}|\langle u_j,2\widetilde D_i\widetilde Y_i\rangle|
&\leq
2{\color{black}b_{d,n}}.
\end{align*}
Apply the finite ghost expansion in Lemma~\ref{lem:finite-ghost-completion} and the order-$Q$ Taylor expansion to each of the three direction patterns. At ghost depth $k$ and Taylor degree $h$, the nonterminal derivative order is $q=4+2k+h<Q$. The repeated pair is $2D_i'Y_i',2D_i'Y_i'$ in the first pattern, two copies of $2\widetilde D_i\widetilde Y_i$ in the second, and two copies of $2D_iY_i$ in the third. Since the Gaussian variance is $(1+s)/2$, two Gaussian integrations by parts cost at most the factor $2$.

We verify that the common-base construction applies separately to all
three patterns. Fix \(s\in[0,1]\), and put
\[
K_Q:=\frac{Q-4}{2},
\qquad
\xi_i:=2\sqrt{1-s}D_iY_i.
\]
Let
\[
\mathcal F_{\mathrm{master}}
:=
\sigma\!\left(
G,G_1,
(D_j,Y_j,D_j',Y_j',
 \widetilde D_j,\widetilde Y_j)_{1\leq j\leq n}
\right).
\]
On a further product extension of the probability space, choose the
finite array
\[
\left(
\xi_{i,p}^{(h)}
\right)_{
\substack{
1\leq i\leq n,\;
1\leq p\leq3,\;
0\leq h\leq K_Q
}}
\]
so that all random vectors in this array are jointly independent, the
entire array is independent of
\(\mathcal F_{\mathrm{master}}\), and $\mathcal L\!\left(\xi_{i,p}^{(h)}\right)
=
\mathcal L(\xi_i)$ for every \(i,p,h\).

For the three direction patterns, set
\[
\begin{aligned}
\mathcal A_i^{(1)}
&:=\sigma(D_i,Y_i,D_i',Y_i'),\\
\mathcal A_i^{(2)}
&:=\sigma(D_i,Y_i,\widetilde D_i,\widetilde Y_i),\\
\mathcal A_i^{(3)}
&:=\sigma(D_i,Y_i).
\end{aligned}
\]
For \(p=1,2,3\), define
\[
H_{i,s}^{(p)}
:=
R_{i,s}+\xi_{i,p}^{(0)}.
\]
Then
\[
H_{i,s}^{(p)}\stackrel d=T_s,
\]
and \(H_{i,s}^{(p)}\) is independent of the joint sigma-field
\[
\mathcal A_i^{(p)}
\vee
\sigma\!\left(
\xi_{i,p}^{(h)}:1\leq h\leq K_Q
\right).
\]

Let \(P_{i,s}\) denote convolution in the spatial variable by
\(\mathcal L(\xi_i)\). When \(P_{i,s}\) is applied to a function
depending on the active directions, those directions are held fixed;
equivalently, the convolution is applied conditionally on
\(\mathcal A_i^{(p)}\). For fixed \(t\),
Lemma~\ref{lem:finite-ghost-completion} is therefore applied
conditionally on \(\mathcal A_i^{(p)}\) to
\begin{align*}
D^4\pi_{\beta,a_0,a}(x+t\xi_i)[\mathcal D_i^{(p)}],
\qquad p=1,2,3,
\end{align*}
where
\begin{align*}
\mathcal D_i^{(1)}
&=[2D_iY_i,2D_iY_i,2D_i'Y_i',2D_i'Y_i'],\\
\mathcal D_i^{(2)}
&=[2D_iY_i,2\widetilde D_i\widetilde Y_i,2\widetilde D_i\widetilde Y_i,2\widetilde D_i\widetilde Y_i],\\
\mathcal D_i^{(3)}
&=[(2D_iY_i)^{\otimes4}].
\end{align*}

For each \(p=1,2,3\), choose, on a further product extension,
a random vector \(H_s^{(p)}\) such that
\[
H_s^{(p)}\stackrel d=T_s.
\]
Choose \(H_s^{(1)},H_s^{(2)},H_s^{(3)}\) jointly independent of
\(\mathcal F_{\mathrm{master}}\) and of the entire ghost array
\[
\left(
\xi_{i,p}^{(h)}
\right)_{
\substack{
1\leq i\leq n,\;
1\leq p\leq3,\;
0\leq h\leq K_Q
}}.
\]

For \(1\leq i\leq n\) and \(1\leq p\leq3\), put
\[
\Gamma_{i,p}
:=
\left(
\mathcal D_i^{(p)},
(\xi_{i,p}^{(h)})_{1\leq h\leq K_Q}
\right).
\]
The preceding independence statements imply that, for every integrable
measurable function \(\Phi\),
\[
\mathbb E\!\left[
\Phi\!\left(H_{i,s}^{(p)},\Gamma_{i,p}\right)
\right]
=
\mathbb E\!\left[
\Phi\!\left(H_s^{(p)},\Gamma_{i,p}\right)
\right].
\]
Consequently, after the ghost and Taylor expansions, every occurrence
of \(H_{i,s}^{(p)}\) in the \(i\)th row expectation may be replaced by
the same common state \(H_s^{(p)}\). By linearity of expectation, this
replacement is made before summing over \(i\). Thus, for each fixed
pattern \(p\), the entire direction array is independent of one common
base state, as required by
Lemma~\ref{lem:common-base-repeated-pair}. At depth \(k=0\), select the repeated pair indicated in the
preceding paragraph. At depth \(k\geq1\), select the repeated ghost
pair $\xi_{i,p}^{(1)},\xi_{i,p}^{(1)}$. Besides the covariance budgets already displayed, equality in law gives
\[
\sum_{i=1}^n
\mathbb E\!\left[
\xi_{i,p}^{(h)}
\bigl(\xi_{i,p}^{(h)}\bigr)^\top
\right]
=
\frac{1-s}{2}I_r
\preceq
\frac12I_r
\]
and
\[
\max_{j\leq2d}
\left|
\left\langle u_j,\xi_{i,p}^{(h)}\right\rangle
\right|
\leq
2b_{d,n}
\qquad\text{almost surely}
\]
for every \(1\leq p\leq3\) and \(0\leq h\leq K_Q\).

The local shift for pattern \(p\) at ghost depth \(k\) is
\begin{align*}
L_{i,k}^{G,p}
&=
t\xi_i+
\sum_{\ell=1}^k
v_\ell\xi_{i,p}^{(\ell)},
\\
\max_{j\leq2d}
\left|
\left\langle u_j,L_{i,k}^{G,p}\right\rangle
\right|
&\leq
2(k+1)b_{d,n}.
\end{align*}
For each \(p\), realize the recoupled common state as
\[
H_s^{(p)}
=
\overline H_s^{(p)}+\sigma_sG^{(p)},
\qquad
\sigma_s^2=\frac{1+s}{2}\geq\frac12,
\]
where \(G^{(p)}\sim N(0,I_r)\),
\[
\overline H_s^{(p)}
\stackrel d=
2\sqrt{1-s}\sum_{\ell=1}^nD_\ell Y_\ell,
\]
and \(\overline H_s^{(p)}\) and \(G^{(p)}\) are independent of each
other and of all direction and ghost arrays. Applying
Lemma~\ref{lem:common-base-repeated-pair} with
\(H=\overline H_s^{(p)}\) and \(G=G^{(p)}\), followed by the same
factorial calculation as in
Lemma~\ref{lem:nonterminal-common-base}, bounds the nonterminal
contribution of order \(q\) for each of the three patterns by
\begin{align*}
Cb_{d,n}^2\beta^2\log(4d)
q^5(C^*q\beta b_{d,n})^{q-4}.
\end{align*}
For the terminal Taylor remainder, put $m_k=4+2k$ and $p_k=Q-m_k$. Exactly as in the explicit calculation in Lemma~\ref{lem:terminal-order-Q}, its absolute value for one row, depth, and pattern is bounded by
\begin{align*}
\frac{C_0^QQ!}{p_k!}\beta^Q
(2b_{d,n})^{m_k}
\{2(k+1)b_{d,n}\}^{p_k}
&\leq
(C^*Q\beta b_{d,n})^Q.
\end{align*}
The terminal ghost term has \(Q\) bounded directions and obeys the
same estimate. Summing over three patterns, fewer than \(Q\) depths,
and \(n\) rows shows that the total absolute terminal contribution is
bounded by
\begin{align}
CnQ^2(C^*Q\beta b_{d,n})^Q.
\label{eq:T-to-G-terminal-bound}
\end{align}
Thus the second-path contraction and terminal bounds have now been verified with the path-specific omitted increment, common state, repeated pair, and covariance budget required here.

The same selector-derivative, Gumbel-cell, and convex-partition calculation as in Lemma~\ref{lem:nonterminal-common-base} yields the upper bound of the nonterminal contribution of order $q$
\begin{align*}
C{\color{black}b_{d,n}}^2\beta^2\log(4d)q^5(C^*q\beta {\color{black}b_{d,n}})^{q-4}.
\end{align*}
Hence
\begin{align*}
\sum_{q=4}^{Q-1}
C{\color{black}b_{d,n}}^2\beta^2\log(4d)q^5(C^*q\beta {\color{black}b_{d,n}})^{q-4}
&\leq
C{\color{black}b_{d,n}}^2\beta^2\log(4d).
\end{align*}
Combining the preceding nonterminal estimate with
Eq.~\eqref{eq:T-to-G-terminal-bound}, and observing that the generator
coefficients and the \(t\)-integration weights are bounded by
universal constants, proves the result.
\end{proof}

\section{Matrix-weighted Gaussian surface bound}
\label{sec:surface-appendix}

\begin{lemma}
\label{lem:unbounded-Gauss-Green}
Let \(r\geq1\), and let \(P\subseteq\mathbb R^r\) be a
full-dimensional polyhedron with facets \(F_1,\ldots,F_{m_P}\) and
outward unit normals \(u_1,\ldots,u_{m_P}\). Let
\(V:\mathbb R^r\to\mathbb R^r\) be continuously differentiable.
Suppose that there exist \(C>0\) and \(k\in\mathbb N_0\) such that,
for every \(x\in\mathbb R^r\),
\begin{align*}
\|V(x)\|_2+|\operatorname{div}V(x)|
&\leq
C(1+\|x\|_2^k)\varphi_r(x).
\end{align*}
Then
\begin{align*}
\int_P\operatorname{div}V(x)\,dx
&=
\sum_{j=1}^{m_P}
\int_{F_j}V(x)^\top u_j\,d\mathcal H^{r-1}(x).
\end{align*}
\end{lemma}

\begin{proof}
Apply the Gauss-Green formula on the bounded Lipschitz domain $\operatorname{int}(P)\cap B_R$; see \citet{evans-gariepy-2015}. Its boundary consists, modulo $\mathcal H^{r-1}$-null intersections, of the flat pieces $F_j\cap B_R$ and the artificial spherical piece $P\cap\partial B_R$. The absolute spherical flux is bounded by
\begin{align*}
C\int_{\partial B_R}(1+\|x\|_2^k)\varphi_r(x)d\mathcal H^{r-1}(x)
&\leq
C_rR^{r-1}(1+R^k)e^{-R^2/2}
\longrightarrow0.
\end{align*}
The volume and flat-facet integrals converge by dominated convergence, proving the identity.
\end{proof}

\begin{proposition}
\label{prop:matrix-surface}
There exists a universal constant \(C>0\) such that the following
holds. Let \(r\ge1\) and \(m_P\ge1\), and let $P
=
\bigcap_{j=1}^{m_P}
\left\{
x\in\mathbb R^r:
\langle u_j,x\rangle\le a_j
\right\}$ be a nonempty, proper, full-dimensional polyhedron with an
irredundant representation, where \(u_1,\ldots,u_{m_P}\) are unit
outward normals. Let \(B_1,\ldots,B_{m_P}\) be symmetric matrices. Then
\begin{align}
\left|
\sum_{j=1}^{m_P}\int_{F_j}\operatorname{tr}\{(xx^\top-I_r)B_j\}\varphi_r(x)d\mathcal H^{r-1}(x)
\right|
&\leq
C\left(\max_{j\leq m_P}\|B_j\|_{\mathrm{op}}\right)\{\log(2m_P)\}^{3/2}.
\label{eq:matrix-surface-bound}
\end{align}
No boundedness assumption is imposed on \(P\).
\end{proposition}

\paragraph{Proof roadmap.}
The proof separates the contribution of each facet into five scalar quantities. Lemma~\ref{lem:facet-decomposition} writes the matrix-weighted integral over $F_j$ in terms of the Gaussian surface mass $s_j$, the threshold $a_j$, the conditional mean $\mu_j$, and the conditional covariance deficit $\Delta_j$. Lemma~\ref{lem:surface-scalar-budgets} controls the total surface, threshold, entropy, and mean terms by $\{\log(2m_P)\}^{3/2}$ using Nazarov's inequality, the entropy variational formula, and Gaussian tail decay.

The remaining term is $\sum_js_j\Delta_j$. Lemma~\ref{lem:ridge-identity} applies Gauss-Green inside each tangent section $Q_j$ and expresses this covariance deficit as a signed sum over codimension-two ridges. Lemma~\ref{lem:normal-strata} associates each ridge with its normal-projection stratum; these strata are disjoint, and their Gaussian masses share the same tangential factor as the ridge surface masses. Lemma~\ref{lem:local-cone-mass} supplies the two-dimensional cone probability needed to compensate for nearly parallel normals. The resulting ridge budget is proved in Lemma~\ref{lem:ridge-budget}. Summing the facet decomposition and these scalar budgets completes the proposition. Unbounded polyhedra are covered by Lemma~\ref{lem:unbounded-Gauss-Green}; no artificial spherical facet remains in the final facet count.

\begin{lemma}
\label{lem:facet-decomposition}
Suppose $\max_{j\leq m_P}\|B_j\|_{\mathrm{op}}\leq1$. Then, for every facet $F_j$,
\begin{align}
\int_{F_j}\operatorname{tr}\{(xx^\top-I_r)B_j\}\varphi_r(x)d\mathcal H^{r-1}(x)
&=
s_j\left[\alpha_j(a_j^2-1)+2a_jv_j^\top\mu_j+\operatorname{tr}\{C_j(K_j-I_{E_j}+\mu_j\mu_j^\top)\}\right],
\label{eq:facet-decomposition}
\end{align}
and its absolute value is bounded by
\begin{align*}
s_j\{1+a_j^2+2|a_j|\|\mu_j\|_2+\|\mu_j\|_2^2+\Delta_j\}.
\end{align*}
\end{lemma}

\begin{proof}
Every $x\in F_j$ has the unique representation $x=a_ju_j+y$ with $y\in Q_j\subseteq E_j$, and
\begin{align*}
\varphi_r(a_ju_j+y)
&=
\varphi_1(a_j)\varphi_{E_j}(y).
\end{align*}
Relative to $\operatorname{span}(u_j)\oplus E_j$,
\begin{align*}
xx^\top-I_r
&=
\begin{pmatrix}
a_j^2-1&a_jy^\top\\
a_jy&yy^\top-I_{E_j}
\end{pmatrix}.
\end{align*}
Integrating the four blocks over $Q_j$ gives
\begin{align*}
\int_{Q_j}y\varphi_{E_j}(y)dy
=
p_j\mu_j,
\quad
\int_{Q_j}yy^\top\varphi_{E_j}(y)dy
=
p_j(K_j+\mu_j\mu_j^\top),
\end{align*}
which proves Eq.~\eqref{eq:facet-decomposition}. Lemma~\ref{lem:convex-conditioning} gives $K_j\preceq I_{E_j}$, so
\begin{align*}
|\operatorname{tr}\{C_j(K_j-I_{E_j}+\mu_j\mu_j^\top)\}|
&\leq
\operatorname{tr}(I_{E_j}-K_j)+\|\mu_j\|_2^2
=
\Delta_j+\|\mu_j\|_2^2.
\end{align*}
The block norms of $B_j$ are at most one, giving the stated bound.
\end{proof}

\begin{lemma}
\label{lem:surface-scalar-budgets}
There exists a universal constant \(C>0\) such that, for every
polyhedral surface configuration in the sense of
Definition~\ref{def:polyhedral-surface-configuration}, all of the
following inequalities hold:
\begin{align*}
\sum_{j=1}^{m_P}s_j \leq
C\sqrt{\log(2m_P)},
\quad
\sum_{j=1}^{m_P}s_j\log\left(\frac{e}{p_j}\right)
\leq
C\{\log(2m_P)\}^{3/2},
\end{align*}
\begin{align*}
\sum_{j=1}^{m_P}s_ja_j^2
\leq
C\{\log(2m_P)\}^{3/2},
\quad
\sum_{j=1}^{m_P}s_j\|\mu_j\|_2^2
\leq
C\{\log(2m_P)\}^{3/2},
\quad
\sum_{j=1}^{m_P}s_j|a_j|\|\mu_j\|_2
\leq
C\{\log(2m_P)\}^{3/2}.
\end{align*}
\end{lemma}

\begin{proof}
For $\varepsilon\geq0$, let $P^\varepsilon=\bigcap_j\{x:\langle u_j,x\rangle\leq a_j+\varepsilon\}$. The Gaussian first-variation formula, obtainable from the coarea formula; see {\color{black}\citet{evans-gariepy-2015}}, identifies the right derivative of $\gamma_r(P^\varepsilon)$ at zero with the total Gaussian surface mass. Nazarov's inequality then gives
\begin{align*}
\sum_{j=1}^{m_P}s_j
&=
\left.\frac{d}{d\varepsilon}\gamma_r(P^\varepsilon)\right|_{0+}
\leq
C\sqrt{\log(2m_P)}.
\end{align*}
For the conditional mean, the entropy variational formula gives, for every $t\in E_j$,
\begin{align*}
t^\top\mu_j
&\leq
\frac12\|t\|_2^2+\log(1/p_j).
\end{align*}
Taking $t=\mu_j$ yields
\begin{align*}
\|\mu_j\|_2^2
&\leq
2\log(1/p_j).
\end{align*}
Since $s_j=\varphi_1(a_j)p_j\leq(2\pi)^{-1/2}p_j$, $p_j\geq\sqrt{2\pi}s_j$. Then, by concavity of $x\mapsto x\log(C/x)$,
\begin{align*}
\sum_js_j\log\left(\frac{e}{p_j}\right)
&\leq
\sum_js_j\log\left(\frac{C}{s_j}\right)
\leq
\left(\sum_js_j\right)\log\left\{\frac{Cm_P}{\sum_js_j}\right\}
\leq
C\{\log(2m_P)\}^{3/2}.
\end{align*}
This and the conditional-mean bound imply
\begin{align*}
\sum_js_j\|\mu_j\|_2^2
&\leq
C\{\log(2m_P)\}^{3/2}.
\end{align*}
On $|a_j|\leq\sqrt{8\log(2m_P)}$,
\begin{align*}
\sum_{|a_j|\leq R}s_ja_j^2
&\leq
8\log(2m_P)\sum_js_j
\leq
C\{\log(2m_P)\}^{3/2}.
\end{align*}
On $|a_j|>\sqrt{8\log(2m_P)}$, $s_j\leq\varphi_1(a_j)$ and $t^2\varphi_1(t)$ is decreasing for $t\geq R$, so
\begin{align*}
\sum_{|a_j|>R}s_ja_j^2
&\leq
8m_P\log(2m_P)\varphi_1\{\sqrt{8\log(2m_P)}\}
\leq
C.
\end{align*}
Thus $\sum_js_ja_j^2\leq C\{\log(2m_P)\}^{3/2}$, and Cauchy-Schwarz gives
\begin{align*}
\sum_js_j|a_j|\|\mu_j\|_2
&\leq
\left(\sum_js_ja_j^2\right)^{1/2}
\left(\sum_js_j\|\mu_j\|_2^2\right)^{1/2}
\leq
C\{\log(2m_P)\}^{3/2}.
\end{align*}
\end{proof}

\begin{lemma}
\label{lem:ridge-identity}
For every facet $F_j$,
\begin{align*}
p_j\Delta_j
&=
p_j\|\mu_j\|_2^2+
\sum_{k:F_j\cap F_k\text{ is a ridge}}d_{k\mid j}r_{jk}^{(j)}.
\end{align*}
Moreover, $\varphi_1(a_j)r_{jk}^{(j)}=\rho_{jk}$, every codimension-two ridge belongs to exactly two facets, and
\begin{align}
\sum_{j=1}^{m_P}s_j\Delta_j
&=
\sum_{j=1}^{m_P}s_j\|\mu_j\|_2^2
+
\sum_{\{j,k\}}(d_{k\mid j}+d_{j\mid k})\rho_{jk}.
\label{eq:global-ridge-identity}
\end{align}
\end{lemma}

\begin{proof}
Apply Lemma~\ref{lem:unbounded-Gauss-Green} on $Q_j\subseteq E_j$ to the tangent vector field $y\varphi_{E_j}(y)$. Since
\begin{align*}
\operatorname{div}\{y\varphi_{E_j}(y)\}
&=
\{\dim(E_j)-\|y\|_2^2\}\varphi_{E_j}(y),
\end{align*}
the volume integral is
\begin{align*}
\int_{Q_j}\{\dim(E_j)-\|y\|_2^2\}\varphi_{E_j}(y)dy
&=
p_j\{\Delta_j-\|\mu_j\|_2^2\}.
\end{align*}
On the boundary induced by $F_k$, $\langle y,n_{k\mid j}\rangle=d_{k\mid j}$, and hence the boundary integral equals $\sum_kd_{k\mid j}r_{jk}^{(j)}$. This proves the first identity. Orthogonal Gaussian factorization on the affine ridge gives
\begin{align*}
\varphi_1(a_j)r_{jk}^{(j)}
&=
\rho_{jk}.
\end{align*}
A codimension-two face has a two-dimensional pointed normal cone, and a pointed polyhedral cone in two dimensions has exactly two extreme rays. Under irredundancy, these rays correspond to the two containing facets. Multiplying the first identity by $\varphi_1(a_j)$ and summing over $j$ proves Eq.~\eqref{eq:global-ridge-identity}.
\end{proof}

\begin{lemma}
\label{lem:normal-strata}
The ridge strata $\mathcal N_{jk}$ are pairwise disjoint modulo Gaussian-null boundaries, and
\begin{align*}
\sum_{\{j,k\}}q_{jk}
&\leq
1.
\end{align*}
Moreover,
\begin{align*}
\rho_{jk}
=
p_{jk}\varphi_2(z_{jk}),
\quad
q_{jk}
=
p_{jk}\int_{z_{jk}+C_{jk}}\varphi_2(w)dw,
\end{align*}
and
\begin{align*}
d_{k\mid j}+d_{j\mid k}
&=
2\sin(\alpha_{jk}/2)\langle b_{jk},z_{jk}\rangle.
\end{align*}
\end{lemma}

\begin{proof}
For a closed convex polyhedron, $x=\Pi_P(y)$ if and only if $y-x$ belongs to the normal cone at $x$. Every $x\in P$ lies in the relative interior of a unique face, and the normal cone is constant on that relative interior; see \citet{rockafellar-wets-1998}. Therefore the sets
\begin{align*}
\operatorname{relint}(F)+\operatorname{relint}N_P(F)
\end{align*}
are disjoint for distinct faces, modulo lower-dimensional boundaries. This proves the stratum disjointness and the probability bound.

The orthogonal decomposition $\mathbb R^r=H_{jk}\oplus N_{jk}$ parametrizes the ridge by $z_{jk}+t$, $t\in T_{jk}$, and the stratum by $z_{jk}+t+v$, $t\in T_{jk}$ and $v\in C_{jk}$. Gaussian density and Hausdorff measure factor under this decomposition, giving
\begin{align*}
\rho_{jk}
&=
\left\{\int_{T_{jk}}\varphi_{H_{jk}}(t)dt\right\}\varphi_2(z_{jk})
=
p_{jk}\varphi_2(z_{jk}),
\end{align*}
and
\begin{align*}
q_{jk}
&=
\left\{\int_{T_{jk}}\varphi_{H_{jk}}(t)dt\right\}
\int_{z_{jk}+C_{jk}}\varphi_2(w)dw
=
p_{jk}\int_{z_{jk}+C_{jk}}\varphi_2(w)dw.
\end{align*}
Finally, substituting the definitions of $d_{k\mid j}$ and $d_{j\mid k}$ and resolving $z_{jk}$ in the plane $N_{jk}$ gives
\begin{align*}
d_{k\mid j}+d_{j\mid k}
&=
2\sin(\alpha_{jk}/2)\left\langle\frac{u_j+u_k}{2\cos(\alpha_{jk}/2)},z_{jk}\right\rangle.
\end{align*}
\end{proof}

\begin{lemma}
\label{lem:local-cone-mass}
There exists a universal constant $c>0$ such that, for every two-dimensional cone $C$ with opening angle $\alpha\in(0,\pi)$ and every $z\in\mathbb R^2$,
\begin{align*}
\int_{z+C}\varphi_2(w)dw
&\geq
c\frac{\sin(\alpha/2)\varphi_2(z)}{(1+\|z\|_2)^2}.
\end{align*}
\end{lemma}

\begin{proof}
Put $\varepsilon=(1+\|z\|_2)^{-1}$. For $h\in C\cap B(0,\varepsilon)$,
\begin{align*}
\frac{\varphi_2(z+h)}{\varphi_2(z)}
&=
\exp\{-\langle z,h\rangle-\|h\|_2^2/2\}
\geq
e^{-3/2}.
\end{align*}
The sector $C\cap B(0,\varepsilon)$ has area
\begin{align*}
\frac{\alpha\varepsilon^2}{2}
&\geq
\sin(\alpha/2)\varepsilon^2.
\end{align*}
Integrating over this sector proves the bound.
\end{proof}

\begin{lemma}
\label{lem:ridge-budget}
There exists a universal constant \(C>0\) such that, for every
polyhedral surface configuration in the sense of
Definition~\ref{def:polyhedral-surface-configuration} with \(r\geq2\),
\begin{align*}
\sum_{\substack{\{j,k\}\\
F_j\cap F_k\ {\rm is\ a\ ridge}}}
(d_{k\mid j}+d_{j\mid k})_+\rho_{jk}
&\leq
C\{\log(2m_P)\}^{3/2},
\\
\sum_{j=1}^{m_P}s_j\Delta_j
&\leq
C\{\log(2m_P)\}^{3/2}.
\end{align*}
The first sum ranges over unordered adjacent facet pairs.
\end{lemma}

\begin{proof}
Lemmas~\ref{lem:normal-strata} and~\ref{lem:local-cone-mass} give
\begin{align*}
(d_{k\mid j}+d_{j\mid k})_+\rho_{jk}
&\leq
C(1+\|z_{jk}\|_2)^3q_{jk}.
\end{align*}
On $\|z_{jk}\|_2\leq\sqrt{8\log(2m_P)}$, stratum disjointness gives
\begin{align*}
\sum_{\|z_{jk}\|_2\leq R}(d_{k\mid j}+d_{j\mid k})_+\rho_{jk}
&\leq
C\{1+\sqrt{8\log(2m_P)}\}^3\sum_{\{j,k\}}q_{jk}
\leq
C\{\log(2m_P)\}^{3/2}.
\end{align*}
On $\|z_{jk}\|_2>\sqrt{8\log(2m_P)}$,
\begin{align*}
(d_{k\mid j}+d_{j\mid k})_+
\leq
2\|z_{jk}\|_2,
\quad
\rho_{jk}
\leq
\varphi_2(z_{jk}).
\end{align*}
There are at most $m_P(m_P-1)/2$ ridges, and $x\mapsto xe^{-x^2/2}$ is decreasing for $x\geq1$, so
\begin{align*}
\sum_{\|z_{jk}\|_2>R}(d_{k\mid j}+d_{j\mid k})_+\rho_{jk}
&\leq
Cm_P^2\sqrt{8\log(2m_P)}(2m_P)^{-4}
\leq
C.
\end{align*}
This proves the first assertion. The ridge sum in Eq.~\eqref{eq:global-ridge-identity} is signed and hence is bounded above by its positive part. Lemma~\ref{lem:surface-scalar-budgets} therefore gives
\begin{align*}
\sum_js_j\Delta_j
&\leq
\sum_js_j\|\mu_j\|_2^2
+
\sum_{\{j,k\}}(d_{k\mid j}+d_{j\mid k})_+\rho_{jk}
\leq
C\{\log(2m_P)\}^{3/2}.
\end{align*}
\end{proof}

We now complete the proof of Proposition~\ref{prop:matrix-surface}.

\begin{proof}
By homogeneity, assume $\max_j\|B_j\|_{\mathrm{op}}\leq1$. If $r=1$, an irredundant one-dimensional polyhedron has at most two boundary points and
\begin{align*}
\sum_j|a_j^2-1|\varphi_1(a_j)
&\leq
C.
\end{align*}
Suppose $r\geq2$. Lemma~\ref{lem:facet-decomposition} gives
\begin{align*}
\left|
\sum_{j=1}^{m_P}\int_{F_j}\operatorname{tr}\{(xx^\top-I_r)B_j\}\varphi_r(x)d\mathcal H^{r-1}(x)
\right|
&\leq
\sum_js_j\{1+a_j^2+2|a_j|\|\mu_j\|_2+\|\mu_j\|_2^2+\Delta_j\}.
\end{align*}
Applying Lemma~\ref{lem:surface-scalar-budgets} to the first four terms and Lemma~\ref{lem:ridge-budget} to the final term gives
\begin{align*}
\left|
\sum_{j=1}^{m_P}\int_{F_j}\operatorname{tr}\{(xx^\top-I_r)B_j\}\varphi_r(x)d\mathcal H^{r-1}(x)
\right|
&\leq
C\{\log(2m_P)\}^{3/2}.
\end{align*}
Lemma~\ref{lem:unbounded-Gauss-Green} justifies every divergence identity for unbounded $P$. Restoring the factor $\max_j\|B_j\|_{\mathrm{op}}$ proves Eq.~\eqref{eq:matrix-surface-bound}.
\end{proof}

\begin{lemma}
\label{lem:skewness-operator}
For each defining normal $u_j$, let $(B_j)_{a_1a_2}
= \sum_{a_3=1}^ru_{ja_3}\sum_{i=1}^n\mathbb E[Y_{ia_1}Y_{ia_2}Y_{ia_3}]$, for $1\leq a_1,a_2\leq r$. Then
\begin{align*}
\max_{j\leq2d}\|B_j\|_{\mathrm{op}}
&\leq
{\color{black}b_{d,n}}.
\end{align*}
\end{lemma}

\begin{proof}
For every unit vector $v\in\mathbb R^r$,
\begin{align*}
v^\top B_jv
&=
\sum_{i=1}^n\mathbb E[\langle u_j,Y_i\rangle\langle v,Y_i\rangle^2].
\end{align*}
Using $|\langle u_j,Y_i\rangle|\leq {\color{black}b_{d,n}}$ and $\sum_i\mathbb E[Y_iY_i^\top]=I_r$ gives
\begin{align*}
|v^\top B_jv|
&\leq
{\color{black}b_{d,n}}\sum_{i=1}^n\mathbb E\langle v,Y_i\rangle^2
=
{\color{black}b_{d,n}}.
\end{align*}
Taking the supremum over unit $v$ proves the claim.
\end{proof}

\begin{lemma}
\label{lem:third-order-skewness}
Under the normalized probabilistic and geometric setup of
Proposition~\ref{prop:T-to-G}, there exists a universal constant
\(C>0\) such that, for every $s\in[0,1]$, $\beta>0$, $a_0\in\mathbb R$ and $a\in\mathbb R^{2d}$, \begin{align*}
\left|
\sum_{a_1,a_2,a_3=1}^r
\left\{
\sum_{i=1}^n
\mathbb E[
Y_{ia_1}Y_{ia_2}Y_{ia_3}
]
\right\}
\mathbb E[
\partial_{a_1a_2a_3}
\pi_{\beta,a_0,a}(T_s)
]
\right|
&\leq
C b_{d,n}\{\log(4d)\}^{3/2}.
\end{align*}
\end{lemma}

\begin{proof}
Write
\begin{align*}
H_s
=
2\sqrt{1-s}\sum_{i=1}^nD_iY_i,
\quad
\sigma_s^2
=
\frac{1+s}{2}.
\end{align*}
Then $T_s\stackrel d=H_s+\sigma_sG$, where $G\sim N(0,I_r)$ is independent of $H_s$, and $\sigma_s^2\geq1/2$. Conditional on $H_s$, third-order Gaussian integration by parts gives
\begin{align*}
\mathbb E_G[\partial_{a_1a_2a_3}\pi_{\beta,a_0,a}(H_s+\sigma_sG)]
&=
\sigma_s^{-3}\mathbb E_G[H_{3,a_1a_2a_3}(G)\pi_{\beta,a_0,a}(H_s+\sigma_sG)].
\end{align*}
By Lemma~\ref{lem:Gumbel-cells}, conditional on $H_s$ and the Gumbel offsets, the dummy-win event is a closed polyhedron in the Gaussian variable with at most $2d$ defining normals. Because every coordinate of $H_3(G)$ is integrable, Fubini's theorem first permits the Gumbel representation to be inserted inside the Hermite expectation. For a fixed realization of $(H_s,\gamma)$, an empty or Gaussian-null winner cell contributes zero. If the cell is all of $\mathbb R^r$, its contribution is also zero because every third Hermite polynomial has mean zero. In every remaining case the cell has positive Gaussian measure and hence is full-dimensional. Removing redundant inequalities gives an irredundant representation with $m_P\leq2d$ facets. Each retained facet comes from one of the original winner inequalities and therefore has outward unit normal equal to one of the original $u_j$; if parallel duplicate inequalities occur, only the tightest one is retained and it carries the same matrix weight $B_j$. Proposition~\ref{prop:matrix-surface} can consequently be applied realization by realization. Its bound is uniform in the random thresholds, and hence averaging over $H_s$ and the Gumbel offsets is justified. Thus it is enough to bound
\begin{align*}
\sum_{a_1,a_2,a_3=1}^r
\left\{\sum_{i=1}^n\mathbb E[Y_{ia_1}Y_{ia_2}Y_{ia_3}]\right\}
\mathbb E[H_{3,a_1a_2a_3}(G)\mathbf1\{G\in P\}].
\end{align*}
Componentwise,
\begin{align*}
\partial_{a_3}\{H_{2,a_1a_2}(x)\varphi_r(x)\}
&=
-H_{3,a_1a_2a_3}(x)\varphi_r(x).
\end{align*}
Applying Lemma~\ref{lem:unbounded-Gauss-Green}, multiplying by the third moment coefficients, and summing over $a_1,a_2,a_3$ gives
\begin{align*}
{}&\sum_{a_1,a_2,a_3=1}^r
\left\{\sum_{i=1}^n\mathbb E[Y_{ia_1}Y_{ia_2}Y_{ia_3}]\right\}
\mathbb E[H_{3,a_1a_2a_3}(G)\mathbf1\{G\in P\}]
\nonumber\\
&=
-\sum_{j=1}^{m_P}\int_{F_j}\operatorname{tr}\{H_2(x)B_j\}\varphi_r(x)d\mathcal H^{r-1}(x).
\end{align*}
Lemma~\ref{lem:skewness-operator} and Proposition~\ref{prop:matrix-surface} bound the absolute value by $C{\color{black}b_{d,n}}\{\log(4d)\}^{3/2}$. Finally, $\sigma_s^{-3}\leq2^{3/2}$, and averaging over $H_s$ and the Gumbel offsets proves the result.
\end{proof}

\section{Supporting results for Section~\ref{sec:summary}}{\color{black}\label{appe:sec:final}}

\begin{proof}[Proof of
Theorem~\ref{thm:subexponential-gaussian-approximation}]
By the reductions leading to
Eq.~\eqref{eq:preprocessing-regime}, the cases $\mathfrak r_{n,d}>c_b$ or $n<N_b$ have already been handled by the trivial probability bound after
increasing \(C_b\). It therefore suffices to work throughout under
\[
\mathfrak r_{n,d}\leq c_b,
\qquad
n\geq N_b,
\]
that is, under Eq.~\eqref{eq:preprocessing-regime}.

In this regime, the construction in
Section~\ref{sec:preparation} gives the normalized envelope
\(b_{d,n}\), and
Lemma~\ref{lem:preprocessing-transfer-main} gives truncation and
Gaussian-transfer error
\[
n^{-20}+C_b n^{-4}.
\]
After increasing \(C_b\) if necessary, we assume \(C_b\geq1\) and
continue to use \(C_b\) for constants depending only on \(b\).

Assumptions~\ref{assu:sub.exp} and~{\color{black}\ref{assu:marginal-variance-lower-bound}} imply
\(B_n\geq b/2\). Indeed, since \(x^2\leq2e^x\) for \(x\geq0\),
\begin{align*}
\mathbb E[X_{ij}^2]
&\leq
2B_n^2\mathbb E\left[\exp\{|X_{ij}|/B_n\}\right]
\leq
4B_n^2.
\end{align*}
If \(b_{d,n}>1/4\), then
\begin{align*}
\frac{B_n}{\sqrt n}\{\log(2dn)\}^{5/2}
&=
\frac{b}{64\sqrt2}
b_{d,n}\{\log(2dn)\}^{3/2}
\end{align*}
is bounded below by a positive constant depending only on \(b\).
Thus, the conclusion again follows from the trivial bound after
increasing \(C_{b}\). We may therefore assume that
\(b_{d,n}\leq1/4\).

Choose \(\delta=b_{d,n}\) in
Proposition~\ref{prop:probability-comparison-through-selectors}.
Propositions~\ref{prop:S-to-T} and~\ref{prop:T-to-G},
Lemma~\ref{lem:conditional-nazarov}, and
Lemma~\ref{lem:preprocessing-transfer-main} give
\begin{align}
\sup_{A\in\mathcal R^d}
\left|
\mathbb P(W\in A)-\mathbb P(Z\in A)
\right|
\leq C_b\Bigg\{
&b_{d,n}
+
\frac{
\{1+\log_+(16d/b_{d,n})\}\sqrt{\log(4d)}
}{\beta}
+
nb_{d,n}^4\beta^4\tau
\nonumber\\
&+
b_{d,n}^2\beta^2\log(4d)\log(1/\tau)
+
b_{d,n}\{\log(4d)\}^{3/2}
\nonumber\\
&+
nQ^2({\color{black}C^*}Q\beta b_{d,n})^Q
+
n^{-20}
+
n^{-4}
\Bigg\}.
\label{eq:overall-bound-original-parameters}
\end{align}

Put
\begin{align*}
L_n
&=
\log(4dn),
\quad
H_n
=
1+\log_+\left(\frac{16dn}{b_{d,n}}\right).
\end{align*}
Then
\begin{align*}
\log(4d)
&\leq
L_n,
\quad
1+\log_+\left(\frac{16d}{b_{d,n}}\right)
\leq
H_n.
\end{align*}
It is enough to consider the case
\begin{align}
b_{d,n}^{2/3}H_nL_n^{2/3}
+
b_{d,n}L_n^{3/2}
&\leq
c_0.
\label{eq:nontrivial-parameter-regime}
\end{align}
Otherwise, the desired conclusion follows from the trivial bound after
increasing the constant.

Moreover,
\begin{align*}
nb_{d,n}^2
=
\left\{
\frac{64\sqrt2 B_n\log(2dn)}{b}
\right\}^2\geq
\left\{
32\sqrt2\log(2dn)
\right\}^2
\geq
1.
\end{align*}
Choose
\begin{align}
\beta
&=
b_{d,n}^{-2/3}L_n^{-1/6},
\quad
\tau
=
\frac{H_nL_n^{4/3}}
{nb_{d,n}^{2/3}},
\quad
Q
=
2\left\lceil A\log(2n)\right\rceil,
\label{eq:corrected-parameter-choice}
\end{align}
where \(A>0\) is chosen below. 
Since \(nb_{d,n}^2\geq1\),
\begin{align*}
\tau
\leq
b_{d,n}^{4/3}H_nL_n^{4/3}\leq
\left(
b_{d,n}^{2/3}H_nL_n^{2/3}
\right)^2.
\end{align*}
Hence, Eq.~\eqref{eq:nontrivial-parameter-regime} implies
\(\tau\leq c_0^2<1/2\). In addition, since
\(H_n\geq1\), \(L_n\geq1\), and \(b_{d,n}\leq1/4\),
$
\tau
\geq
\frac1n$,
and therefore
\begin{align*}
\log(1/\tau)
&\leq
\log n
\leq
CH_n.
\end{align*}

Since \(b_{d,n}\leq1/4\), we have \(H_n\geq L_n\). Therefore,
Eq.~\eqref{eq:nontrivial-parameter-regime} gives
\begin{align*}
\beta b_{d,n}
&=
b_{d,n}^{1/3}L_n^{-1/6}
\leq
c_0^{1/2}L_n^{-1}.
\end{align*}
Also, \(Q\leq CA L_n\), since
\(\log(2n)\leq L_n\). First choose \(A\) sufficiently large that
\begin{align*}
nQ^2 16^{-Q}
&\leq
n^{-4}
\end{align*}
for all sufficiently large \(n\), and then choose \(c_0\) sufficiently
small that
\begin{align*}
{\color{black}C^*}Q\beta b_{d,n}
&\leq
\frac{1}{16}.
\end{align*}
It follows that
\begin{align*}
nQ^2({\color{black}C^*}Q\beta b_{d,n})^Q
&\leq
nQ^2 16^{-Q}
\leq
n^{-4}.
\end{align*}

With the choices in Eq.~\eqref{eq:corrected-parameter-choice},
\begin{align*}
\frac{
\{1+\log_+(16d/b_{d,n})\}\sqrt{\log(4d)}
}{\beta}
&\leq
b_{d,n}^{2/3}H_nL_n^{2/3},\\
nb_{d,n}^4\beta^4\tau
&=
b_{d,n}^{2/3}H_nL_n^{2/3},\\
b_{d,n}^2\beta^2\log(4d)\log(1/\tau)
&\leq
Cb_{d,n}^{2/3}H_nL_n^{2/3},\\
b_{d,n}\{\log(4d)\}^{3/2}
&\leq
b_{d,n}L_n^{3/2}.
\end{align*}
The standalone \(b_{d,n}\) term is absorbed by
\(b_{d,n}L_n^{3/2}\). Since \(B_n\geq b/2\), the terms \(n^{-20}\)
and \(n^{-4}\) are also absorbed by
\(b_{d,n}L_n^{3/2}\) for all sufficiently large \(n\). Consequently,
\begin{align}
\sup_{A\in\mathcal R^d}
\left|
\mathbb P(W\in A)-\mathbb P(Z\in A)
\right|
&\leq
C_b
\min\left\{
1,
b_{d,n}^{2/3}H_nL_n^{2/3}
+
b_{d,n}L_n^{3/2}
\right\}.
\label{eq:bounded-final-rate}
\end{align}

Finally, no polynomial-dimension assumption is needed to control
\(H_n\) and \(L_n\). Indeed,
\begin{align*}
L_n
&=
\log(4dn)
\leq
C\log(2dn),
\end{align*}
while \(B_n\geq b/2\) and the definition of \(b_{d,n}\) give
\begin{align*}
\frac{16dn}{b_{d,n}}
&=
\frac{bd n^{3/2}}
{4\sqrt2B_n\log(2dn)}
\leq
\frac{d n^{3/2}}
{2\sqrt2\log(2dn)}.
\end{align*}
Therefore,
\begin{align*}
H_n+L_n
&\leq
C\log(2dn).
\end{align*}
It follows that
\begin{align*}
b_{d,n}^{2/3}H_nL_n^{2/3}
\leq
C_{b}
\left(\frac{B_n^2}{n}\right)^{1/3}
\{\log(2dn)\}^{7/3},\quad
b_{d,n}L_n^{3/2}
\leq
C_{b}
\frac{B_n}{\sqrt n}
\{\log(2dn)\}^{5/2}.
\end{align*}
Substituting these bounds into
Eq.~\eqref{eq:bounded-final-rate} proves the final bound. Finitely many
remaining values of \(n\) are absorbed into \(C_{b}\).

It remains to pass from the rectangle bound to the critical-value
statement. For every \(a\in\mathbb R^d\) and \(t\in\mathbb R\), define
\[
W_n
:=
\frac{1}{\sqrt n}\sum_{i=1}^nX_i,
\qquad
R_{a,t}
:=
\prod_{j=1}^d(-\infty,t-a_j].
\]
Then
\[
\{T_n(a)\leq t\}
=
W_n^{-1}(R_{a,t}),
\qquad
\{T_n^G(a)\leq t\}
=
Z^{-1}(R_{a,t}).
\]

By Assumption~\ref{assu:marginal-variance-lower-bound},
each Gaussian marginal \(Z_j\) has positive variance. Hence
\(T_n^G(a)=\max_j(Z_j+a_j)\) has no atoms, because, for every
\(t\in\mathbb R\),
\[
\mathbb P\{T_n^G(a)=t\}
\leq
\sum_{j=1}^d
\mathbb P\{Z_j=t-a_j\}
=0.
\]
Thus the distribution function of \(T_n^G(a)\) is continuous.
Applying Lemma~\ref{lem:critical-value-transfer} with
\[
V=T_n(a),
\qquad
V^G=T_n^G(a),
\qquad
q_{1-\alpha}^G=c_{n,1-\alpha}^G(a)
\]
completes the proof.
\end{proof}

\end{document}